\documentclass[11pt]{amsart}   	
\usepackage{geometry}             
\usepackage[dvipsnames]{xcolor}   		             \usepackage{microtype}	
\usepackage{graphicx}				
\usepackage{amssymb,mathrsfs}
\usepackage{amssymb,amsthm,amsmath,stmaryrd}
\usepackage{tikz}
\usepackage{tikz-cd}
\usepackage{tikz-3dplot}
\usepackage{accents,upgreek,enumerate}
\usepackage[headings]{fullpage}
\usepackage{bm}
\usepackage{mathtools}
\usepackage[all]{xy}
\usepackage{caption}
\usepackage{ wasysym }
\usepackage[math]{cellspace}

\newtheorem{innercustomthm}{Theorem}
\newenvironment{customthm}[1]
  {\renewcommand\theinnercustomthm{#1}\innercustomthm}
  {\endinnercustomthm}
  
\usetikzlibrary{calc}
\usetikzlibrary{fadings}
\usetikzlibrary{decorations.pathmorphing}
\usetikzlibrary{decorations.pathreplacing}

\DeclareMathAlphabet{\mathpzc}{OT1}{pzc}{m}{it}

\usepackage[backref=page]{hyperref}
\hypersetup{
  colorlinks   = true,          
  urlcolor     = violet,          
  linkcolor    = violet,          
  citecolor   = blue             
}

\newtheorem{theorem}{Theorem}[subsection]
\newtheorem{corollary}[theorem]{Corollary}
\newtheorem{lemma}[theorem]{Lemma}
\newtheorem{proposition}[theorem]{Proposition}

\newtheorem{quasi-theorem}[theorem]{Quasi-Theorem}

\theoremstyle{definition}
\newtheorem{definition}[theorem]{Definition}
\newtheorem{algorithm}[theorem]{Algorithm}
\newtheorem{warning}[theorem]{Warning}

\newtheorem{assumption}{Assumption}[subsection]

\newtheorem{example}[theorem]{Example}
\newtheorem{blank remark}[theorem]{}

\theoremstyle{remark}
\newtheorem{rem1}[theorem]{Remark}
\newenvironment{remark}{\begin{rem1}\em}{\end{rem1}}

\newtheorem{not1}[theorem]{Notation}

\makeatletter
\def\@tocline#1#2#3#4#5#6#7{\relax
  \ifnum #1>\c@tocdepth 
  \else
    \par \addpenalty\@secpenalty\addvspace{#2}%
    \begingroup \hyphenpenalty\@M
    \@ifempty{#4}{%
      \@tempdima\csname r@tocindent\number#1\endcsname\relax
    }{%
      \@tempdima#4\relax
    }%
    \parindent\z@ \leftskip#3\relax \advance\leftskip\@tempdima\relax
    \rightskip\@pnumwidth plus4em \parfillskip-\@pnumwidth
    #5\leavevmode\hskip-\@tempdima
      \ifcase #1
       \or\or \hskip 1em \or \hskip 2em \else \hskip 3em \fi%
      #6\nobreak\relax
    \dotfill\hbox to\@pnumwidth{\@tocpagenum{#7}}\par
    \nobreak
    \endgroup
  \fi}
\makeatother

\newcommand{\A}{{\mathbb{A}}}           
\newcommand{\CC} {{\mathbb C}}          
            
\newcommand{\NN} {{\mathbb N}}		
\newcommand{\PP}{\mathbb{P}}         
		
\newcommand{\RR} {{\mathbb R}}		

\newcommand{\Rb} {{\mathbf R}}		

\def\setminus{\smallsetminus}

\newcommand{\Hom}{\operatorname{Hom}}

\DeclareMathOperator{\spec}{Spec}

\def\fM{\mathfrak{M}}

\newcommand{\plC}{\scalebox{0.8}[1.3]{$\sqsubset$}}
\newcommand{\plD}{\scalebox{0.85}[1.33]{$\rhd$}}
\newcommand{\bplC}{\scalebox{0.8}[1.3]{$\bm{\sqsubset}$}}

\newcommand{\Kbar}{\mathsf{K}}

\def\trop{\mathrm{trop}}

\newcommand{\Spec}{\operatorname{Spec}}

\makeatletter
\def\blfootnote{\xdef\@thefnmark{}\@footnotetext}
\makeatother

\title[]{Logarithmic Gromov--Witten theory with expansions}

\date{}

\author{Dhruv Ranganathan}

\address{Dhruv Ranganathan \newline Department of Pure Mathematics and Mathematical Statistics\newline Wilberforce Road, University of Cambridge, \newline Cambridge, CB3 0WA UK}
\email{\href{mailto:dr508@cam.ac.uk}{dr508@cam.ac.uk}}

\begin{document}

\begin{abstract}
We construct relative Gromov--Witten theory with expanded degenerations in the normal crossings setting and establish a degeneration formula for the resulting invariants. Given a simple normal crossings pair $(X,D)$, we show that there exist proper moduli spaces of curves in $X$ with prescribed boundary conditions along $D$, equipped with virtual classes. Each point in such a moduli space parameterizes a map from a nodal curve to an expanded degeneration of $X$ that is dimensionally transverse to the strata. In the context of maps to a simple normal crossings degeneration, the virtual fundamental class is known to decompose as a sum over tropical maps. We use the expanded formalism to prove the degeneration formula -- we reconstruct the virtual class attached to a tropical map in terms of spaces of maps to expansions attached to the vertices. 
\end{abstract}

\blfootnote{MSC 2020 Classification: 14N35, 14A21, 14T90}
\blfootnote{Keywords: Logarithmic Gromov--Witten theory, degeneration formula, expanded degenerations}

\maketitle

\vspace{-0.2in}

\setcounter{tocdepth}{1}
\tableofcontents

\noindent
\newpage
\section*{Introduction} 


\subsection{The problem} Let $X$ be a smooth projective algebraic variety and let $D\subset X$ be a simple normal crossings divisor with components $D_1,\ldots,D_k$. Relative Gromov--Witten theory concerns maps of pairs
\[
(C,p_1,\ldots, p_n)\to (X,D),
\]
where $C$ is a smooth genus $g$ curve in a homology class $\beta$, meeting the component $D_i$ at $p_j$ with contact order $c_{ij}\in \mathbb N$. For each $j$, the contact order $c_{ij}$ is nonzero for at most one index $i$. These numerical data are packaged by the symbol $\Gamma$ and the divisor $D$ is implicit in the notation $X$. There is a non-proper Deligne--Mumford stack $\mathsf{K}_\Gamma^\circ(X)$ parameterizing such maps. 

The problem considered in this paper is to build a proper moduli space $\mathsf{K}_\Gamma(X)$ equipped with universal marked curve and target families
\[
\begin{tikzcd}
\mathcal C \arrow{rr} \arrow{dr} & & \mathcal X\arrow{dl} \arrow{r} & X\\
& \mathsf{K}_\Gamma(X), & &
\end{tikzcd}
\]
and sections $s_i:\mathsf{K}_\Gamma(X)\to \mathcal C$. Two requirements are sought.
\begin{enumerate}[(A)]
\item {\bf Virtual toroidality.} The moduli space $\mathsf{K}_\Gamma(X)$ is virtually smooth over an equidimensional logarithmically smooth stack. It is therefore equipped with a virtual fundamental class giving rise to a system of Gromov--Witten invariants of $X$ \textbf{relative} to $D$.
\item {\bf Transversality.} Fix a point $p$ of $\Kbar_\Gamma(X)$ and denote by $(\mathcal C_p,s_i)\to \mathcal X_p$ the map at this point. Then $\mathcal C_p$ meets $\mathcal X_p$ at $s_i$ along a divisor $\mathcal D_{j,p}$ with contact order $c_{ij}$. The divisor $\mathcal D_{j,p}$ maps to $D_{j}$ on projecting to $X$. 
\end{enumerate}

Canonicity of the spaces is not demanded, and we permit discrete choices. However, the resulting Gromov--Witten theory must be independent of choices. The problem is solved in the first part of the paper. The solution leads to a gluing formula for the Gromov--Witten theory of simple normal crossings degenerations, which occupies the second part of the paper. 

\subsection{A brief history} In algebraic geometry this problem was first considered by Li in the case where $D$ is smooth, following work of Gathmann and Vakil~\cite{Gat02,Vak00}. In order to obtain a compactification, he studies maps to \textbf{expansions} of $X$ along $D$, called ``accordions''. These are obtained by iteratively degenerating $X$ to the normal cone of $D$ in $X$, and form the fibers of $\mathcal X$ above. After replacing the target with an accordion, transverse stable limits of families of relative maps exist and are unique. The resulting moduli space $\mathsf{Li}_\Gamma(X)$ is a proper Deligne--Mumford stack with a virtual fundamental class~\cite{Li01}. Li's theory is central to enumerative geometry. The original construction of the virtual class is technical, but the advent of new methods -- orbifolds and logarithmic structures -- offered simplifications~\cite{AF11,Kim08}. Important conceptual advances were made in~\cite{AMW12}.

There have been several attempts to generalize Li's approach to the simple normal crossings setup, but it appears that the idea was abandoned in favour of the logarithmic approach. Remarkably, the transversality condition demanded above can be avoided entirely for the purpose of \textit{defining} relative invariants. In 2001, Siebert suggested that one could study stable maps within the category of logarithmic schemes. This leads to a concise solution to the problem. In the logarithmic category, a map $C\to X$ can have components that are scheme theoretically contained in $D$ and nevertheless carry a well-defined contact order at marked points. Abramovich--Chen and Gross--Siebert constructed  moduli spaces $\mathsf{ACGS}_\Gamma(X)$ of logarithmic maps to a simple normal crossings pair $(X,D)$, equipped with a virtual class~\cite{AC11,Che10,GS13}, building on striking work of Nishinou and Siebert~\cite{NS06}. Crucially, the work of Gross, Nishinou, and Siebert  inserted tropical geometry into the theory of stable maps, and this is fundamental to the direction proposed here.

The logarithmic theory of maps is powerful, but the transversality condition is still desirable. A basic reason is to complete the conceptual parallel with Li's theory, but there are more compelling ones. Many applications of relative Gromov--Witten theory come via the degeneration formula~\cite{AF11,CheDegForm,Li02}. The formula relates spaces of maps with matching boundary conditions with the space of maps to a general fiber. In the logarithmic theory, a parallel degeneration formula has not materialized, though there is encouraging recent work in this direction~\cite{ACGS15,ACGS17,KLR}. The degeneration formula is the second contribution of our paper. In the logarithmic context, a basic difficulty arises when gluing logarithmic structures in families of stable maps. Gluing the curves compatibly with the underlying schematic map to $X$ is necessary but rarely sufficient. The remaining ``purely logarithmic'' gluing presents several challenges. In our framework, the logarithmic data is recorded in a schematic fashion by a factorization of a map to $X$ through an expansion. The role of the logarithmic structures is diminished, and traditional schematic methods are implemented. 


\subsection{Main results} Let $(X,D)$ be a toroidal pair without self-intersections.

\begin{customthm}{A}[{\it Stable maps to expansions}]\label{thm: transverse-maps}
There exists a proper and virtually toroidal moduli space $\mathsf{K}_\Gamma(X)$ equipped with flat source and target families
\[
\begin{tikzcd}
\mathcal C \arrow{rr} \arrow[swap]{dr}{\pi_s} & & \mathcal X\arrow{dl}{\pi_t} \arrow{r} & X\\
& \mathsf{K}_\Gamma(X), & &
\end{tikzcd}
\]
such that the universal curve is transverse to the universal target at every point on the base. The space is equipped with a virtual fundamental class and with (i) a forgetful morphism to the stack of curves $\mathfrak M_{g,n}$ and (ii) evaluation morphisms to the strata of $X$. Gromov--Witten invariants may be defined by pulling back Chow cohomology classes along these morphisms and integrating against the virtual fundamental class. 
\end{customthm}

The reader may consult Example~\ref{ex: line-degeneration}, which exhibits a stable limit of a family in the space of expanded maps, and compares it to the limit as a logarithmic map. At present, we note the following features of the construction.
\begin{enumerate}[(A)]
\item {\bf Target expansions.} The fibers of $\pi_t$ are toroidal degenerations of $X$. Each component of such a degeneration is birational to an equivariant partial compactification of a torus bundle over a stratum of $(X,D)$.
\item {\bf Smooth intersections} Given a fiber $\mathcal X_p$ of $\pi_t$, the intersection of any two irreducible components of $\mathcal X_p$ is a smooth, possibly non-proper divisor in each. 
\item {\bf Comparison morphism.} There is a canonical morphism
\[
\mathsf{K}_\Gamma(X)\to \mathsf{ACGS}_\Gamma(X)
\]
which is a logarithmic modification compatible with virtual structures, and in particular identifies virtual classes, see Section~\ref{sec: vir-birationality}. 
\end{enumerate}

In order to prove this result, we start with the space of logarithmic maps and then use toroidal modifications of the curve and target family to obtain the requisite transversality. We perform semistable reduction for the family by toroidal modifications of the base. These constructions are first made at the combinatorial level, then for maps to the Artin fan, and finally for curves in $X$. 

The construction produces an infinite family of moduli spaces of transverse maps controlled by tropical moduli data. In the main text, these spaces are denoted $\Kbar^\lambda_\Gamma(\cdot)$, with $\lambda$ denoting an auxiliary polyhedral choice. The choices of $\lambda$ form an inverse system, analogous to the inverse system of refinements of a given fan. The virtual classes form a compatible system in Chow homology, see Proposition~\ref{prop: vir-birationality}. The properties in Theorem~\ref{thm: transverse-maps} are true \textit{level-by-level} on a cofinal subsystem, and not only on the inverse limit. 

The transversality condition is sufficiently robust to accommodate a cycle theoretic gluing formula for the Gromov--Witten invariants of degenerations.

\begin{customthm}{B}[{\it Degeneration, decomposition, and gluing}]\label{thm: deg-form}
Let $\mathscr Y\to \A^1$ be a toroidal degeneration without self intersections, general fiber $Y_\eta$, and special fiber $Y_0$. There exist moduli spaces $\Kbar_\Gamma(\cdot)$ of maps to expansions with the following properties. 
\begin{enumerate}[(A)]
\item {\bf Virtual deformation invariance.} There is an equality of virtual classes
\[
[\Kbar_\Gamma(Y_0)] = [\Kbar_\Gamma(Y_\eta)]
\]
in the Chow group of the space of maps $\Kbar_\Gamma(\mathscr Y)$.
\item {\bf Decomposition.} The virtual class of maps to $Y_0$ decomposes as a sum over combinatorial splittings of the discrete data (i.e. tropical stable maps)
\[
[\Kbar_\Gamma(Y_0)] = \sum_\rho m_\rho [\Kbar_\rho(Y_0)]
\]
where $m_\rho\in \mathbb Q$ are explicit combinatorial multiplicities depending on the splitting $\rho$. Each space $\Kbar_\rho(Y_0)$ is a space of maps to expansions, marked by splitting type.
\item {\bf Gluing.} For each splitting $\rho$ with graph type $G$, there are moduli spaces $\Kbar_\rho(X_v)$ of maps to expansions of components $X_v$ of $Y_0$. There is a virtual birational model of their product
\[
{\bigtimes_v} \Kbar_{\rho}(X_v) \to \prod_v \Kbar_{\rho}(X_v),
\]
and an explicit formula relating the virtual class $[\Kbar_\rho(Y_0)]$ with the virtual class $[{\bigtimes_v} \Kbar_{\rho}(X_v)]$ and the class of the relative diagonal of the universal divisor expansion.
\end{enumerate}
\end{customthm}

We refer the reader to the main text for the precise birational modification above. Its genesis is simple. The product of the moduli spaces of maps to $X_v$ encodes disconnected maps to expansions of components of the normalization. If we are to glue two expansions, a bubbling of the target on one side may force a bubbling of the target on the other side\footnote{\textit{The fates of the two sides are tied}, says Pandharipande.} Since the targets on the different factors of the product must interact, a ``criss-cross'' blowup is necessary. 

\subsection{Using the formula} The formula is more complicated than the corresponding double point formula, and it can be nontrivial to extract numerical consequences. However the complexity can be understood concretely and worked with. It arises from the fact that the gluing condition is not imposed by a diagonal condition coming from the strata of $(X,D)$ but rather the strict transform of this class along a birational modification, as described above. The modification is obtained by subdividing the product of tropical moduli spaces. The modification is completely described by piecewise linear geometry on the associated family of tropical curves. Describing the blowup is a finite problem that can be coded. Intersection theory computations with these modifications can be carried out by using the Fulton--Macpherson blowup formula for comparing strict and total transforms~\cite[Section~6.7]{Ful13}. The blowups are logarithmic, so the blowup formula is captured by Chow operators on the Artin fans of these moduli spaces of maps. The rational Chow cohomology of Artin fans is the ring of piecewise polynomials~\cite{MPS20,MR21,Pay06,R20}. We expect these to become a natural calculus for statements in logarithmic intersection theory. 


Even without a systematic calculus, we note that these formulae can often be ``black boxed'' to derive structural results. A prototypical example is provided by work of Pandharipande and Pixton~\cite[Section~1.2 {\it \&} 5]{PP17}. They compare the bi-relative Gromov--Witten and stable pairs theory by working with an instance of the formula above and deduce the Gromov--Witten/pairs correspondence for a large class of varieties. 

There are cases in the existing literature where the formula collapses and these are recorded in the final section of the paper. We also present a simplified formula for genus $0$ invariants in the case of a triple point degeneration, see Section~\ref{sec: implementations}.


\subsection{Further discussion} Our results complete the simple normal crossings generalization of Li's theory. The virtual class of Li's theory also satisfies a remarkable torus localization formula~\cite{GV05,MR19}. The parallel formula in the expanded theory is the natural next step.  If proved, these will yield reconstruction theorems for Gromov--Witten invariants via semistable reduction. A basic test for these developments would be to show that the logarithmic theory of $(X,D)$ is determined by the absolute theory of the strata, generalizing the calculation schemes of Gathmann and Maulik--Pandharipande~\cite{Ga03,MP06}. See~\cite[Corollary~Y]{BNR22} for progress in genus $0$.

A cousin of the theory in the present paper is \textit{punctured} logarithmic Gromov--Witten theory, recently developed by Abramovich--Chen--Gross--Siebert~\cite{ACGS17}. Superficially, maps to expansions are different beasts than punctured maps, but we expect that the gluing formulas will be similar. A version of the expanded theory for non-rigid targets can be constructed, which form the analogue of punctured maps, see~\cite{CN21}. 


The principle underlying our approach is that the category of all logarithmic stable maps, rather than the minimal or basic maps considered in the literature, can be worked with in its own right. The minimality condition amounts to the tropical moduli of a logarithmic family being as large as possible. We vary this condition by allowing the tropical moduli to be supported on faces of a refinement of the tropical moduli space. This leads to subcategories consisting of non-minimal maps that are represented over schemes by logarithmic modifications. The same ideas arise in earlier work of Santos-Parker, Wise, and the author, which is a source of inspiration~\cite{RSW17A,RSW17B}. Although non-archimedean geometry makes no overt appearance here, the ideas are parallel to Raynaud's approach to analytic geometry via admissible formal schemes.

\subsection{Inverse limits for Chow groups and moduli spaces} Our methods produce an infinite family of virtually birational moduli spaces for which a gluing formula holds. The results could be reformulated using an elegant idea of Aluffi~\cite{Alu05}. Given a logarithmic scheme or stack $\mathsf X$, consider the category $\mathcal C_{\mathsf X}$ whose objects are logarithmic modifications $\mathsf{X}'\to \mathsf X$, see~\cite{AW,Kato94}, and whose morphisms are the commuting triangles. Systems $\mathcal C_{\mathsf X}$ and $\mathcal C_{\mathsf Y}$ are equivalent if they contain objects with isomorphic source. The category $\mathcal C_{\mathsf X}$ is equipped with a Chow group defined as the inverse limit of Chow groups of modifications, and an operational Chow ring obtained from the direct limit.  Virtual fundamental classes of the spaces constructed here give a class on the inverse limit, and Gromov--Witten invariants are integrals of operational Chow classes against the virtual class. The gluing formula is then an expression for the virtual class in the Chow groups of these categories. Recent work on logarithmic intersection theory by Barrott realizes these ideas~\cite{Bar18}.

One could try to pass to the inverse limit of the spaces themselves. A natural category in which to do this is that of locally topologically ringed spaces. The limit resembles an object in Foster and Payne's theory of adic tropicalizations~\cite{FP15}. Indeed, this approach was suggested in~\cite[arXiv v1, Remark 4.3.2]{R15b}. One could also sheafify the category of logarithmic schemes in the topology of logarithmically \'etale modifications and formulate the gluing formula on the \textbf{valuativization}~\cite{KatoLDDT}.  The approaches would provide an elegant repackaging of our framework, but we prefer more widely known techniques.



\subsection{Symplectic geometry and exploded manifolds} There is a rich parallel story in the symplectic category due to Li--Ruan, Ionel--Parker, Ionel, Farajzadeh Tehrani--McLean--Zinger, and Farajzadeh Tehrani developing Gromov--Witten theories in relative and logarithmic geometries~\cite{I15,IP03, IP04,LR01,Teh17,TMZ}. The symplectic side was motivated by considerations in Donaldson--Floer theory and predates the algebro-geometric versions. 

In a series of papers, Parker has developed an approach to Gromov--Witten theory relative to normal crossings divisors, using exploded manifolds~\cite{Par11,Par12a}. He proves a gluing formula for Gromov--Witten invariants similar to the one we establish here~\cite{Par17a}. It appears that one can compare Parker's category to logarithmic schemes by base changing to $\spec(\RR_{\geq 0}\to \CC)$, and then passing to the inverse limit as above. We work with the inverse system here, rather than the inverse limit. Our theory appears is formally consistent with exploded Gromov--Witten theory, with several parallels. The approach here is analogous to Parker's \textit{tropical completions}, while punctured Gromov--Witten theory is analogous to \textit{cut curves}, see~\cite[Section~13]{Par12b}. We hope that this paper will lead to a wider understanding of exploded manifolds.

\subsection{User's guide} In the first part of the paper, we study logarithmic maps to expansions. The types of expansions we consider are obtained by polyhedral subdivisions of tropical targets, which are defined at the beginning of Section~\ref{sec: comb-ss-reduction}. The spaces of tropical maps to these expansions with the requisite transversality properties are constructed in Section~\ref{sec: transverse-maps} using toroidal weak semistable reduction~\cite{AK00}. These combinatorial constructions are lifted to statements about logarithmic stable maps using Artin fan techniques in Section~\ref{sec: maps-from-subs}, leading to Theorem~\ref{thm: transverse-maps}. We begin the second part of the paper by studying the degeneration formula in a purely combinatorial setting using extended tropicalizations. The appropriate moduli spaces of transverse maps are constructed in Section~\ref{sec: tropical-gluing}. Special attention is paid to the geometry of the product in Section~\ref{sec: modify-product}. Finally, the combinatorial constructions are used to prove the virtual gluing formula in Section~\ref{sec: general-gluing}, leading to Theorem~\ref{thm: deg-form}.

Our results are stated for toroidal embeddings without self-intersection, rather than all logarithmically smooth targets. In light of advances in logarithmically \'etale descent, this is largely an expositional choice~\cite{ACMW,AW}.  Logarithmic Gromov--Witten theory is insensitive to logarithmic modifications, so this presents no restrictions in applications. 

\subsection{Recent progress} The present paper was first circulated in 2018, and there has been significnant progress in the intervening years. A parallel logarithmic theory of Donaldson--Thomas invariants has been developed, via the construction of a moduli space of higher rank expansions~\cite{MR20}. The methods in loc. cit. can be applied to construct the spaces in Theorem~\ref{thm: transverse-maps} in a manner that is logically independent from the foundational papers in logarithmic Gromov--Witten theory~\cite{AC11,Che10,GS13}, and this is explained by Carocci--Nabijou~\cite{CN21}. On the other hand, punctured Gromov--Witten theory gives rise to a different solution to gluing problems for logarithmic maps~\cite{ACGS17}. The approaches appear to have complementary strengths. Additional progress has been made in logarithmic intersection theory, see~\cite{Bar18,Herr19,MPS20,MR21}. 

\subsection*{Acknowledgements} The strategy in this paper arose from attempting to justify a false claim I made to D. Maulik and I am grateful to him for countless discussions. I thank D. Abramovich and M. Gross for encouragement and interesting discussions along the way. Related ideas were developed in work with J. Wise and K. Santos-Parker~\cite{RSW17A,RSW17B} and I learned a great deal from them. I have benefited from conversations with L. Battistella, D. Bejleri, Q. Chen, R. Cavalieri, M. van Garrel, N. Nabijou, S. Marcus, H. Markwig, S. Molcho, R. Pandharipande, H. Ruddat, and M. Talpo. Finally, I would like to acknowledge the influence of the wonderful paper of Abramovich--Karu~\cite{AK00} on the text. The paper was improved by the comments of an anonymous referee.

\noindent
The author is supported by EPSRC New Investigator Grant EP/V051830/1.

\subsection*{Conventions} We work exclusively with fine and saturated logarithmic schemes that are locally of finite type over $\CC$ that are locally of finite type. All cones will be rational and polyhedral, equipped with a toric monoid of positive linear functions. Given a moduli space $\mathsf{K}$ equipped with a virtual fundamental class, we denote this class by $[\mathsf 
K]$. The logarithmic structure sheaf will be banished from the notation, and the word ``underlying'' will be used when ignoring the logarithmic structure. The symbols $\plC$, $\plD$ and variants, being piecewise linear, will be used to denote tropical curves, in a continuing effort to popularize a convention of Abramovich.

{\Large \part{The expanded theory}}

\vspace{0.2in}

\section{Curves and automorphisms}

\subsection{Logarithmic curves and tropicalizations} We begin by recalling the definition of a tropical curve. 

\begin{definition}\label{def: trop-curve}
An \textbf{$n$-marked tropical curve} $\plC$ or simply a \textbf{tropical curve} is a finite graph $G$ with vertex and edge sets $V$ and $E$, equipped with
\begin{enumerate}
\item a \textbf{marking function} $m: \{1,\ldots,n\}\to V$,
\item a \textbf{genus function} $g:V\to \NN$,
\item a \textbf{length function} $\ell: E\to \RR_{+}$.
\end{enumerate}
The \textbf{genus} of a tropical curve $\plC$ is defined to be 
\[
g(\plC) = h_1(G)+\sum_{v\in V} g(v)
\]
where $h_1(G)$ is the first Betti number of the graph $G$.
\end{definition}

The tropical curve $\plC$ has a \textbf{metric realization}. First endow the graph $G$ with its evident metric from the edge length function.  Then, for each marking $i$ with $m(i) = v$, attach a copy of $\RR_{\geq 0}$ to $v$ at the point $0$.  The point at infinity of this unbounded edge will be understood as ``the'' marked point $p_i$. 

If $\ell$ is allowed to take values in an arbitrary monoid, we obtain families of tropical curves.

\begin{definition}
Let $\sigma$ be a cone with dual cone $S_\sigma$. A \textbf{family of $n$-marked tropical curves over $\sigma$} is a graph $G$ with marking and genus function as in Definition~\ref{def: trop-curve}, and whose length function takes values in $S_\sigma$.
\end{definition}

A point of $\sigma$ is a monoid homomorphism $\varphi: S_\sigma \to \RR_{\geq 0}$. If this homomorphism is applied to the edge length $\ell(e)\in S_\sigma$, we obtain a positive real length for each edge and thus a tropical curve.

\subsection{Logarithmic curves}\label{sec: log-trop} Let $(S,M_S)$ be a logarithmic scheme. A \textbf{family of logarithmically smooth curves over $S$} is a logarithmically smooth, flat, and proper morphism
\[
\pi: (C,M_C) \to (S,M_S),
\]
with connected and reduced fibers of dimension $1$, see~\cite{Kat00}. 

%
%

Associated to a logarithmic curve $C\to S$ is a family of tropical curves, called its \textbf{tropicalization}. Given a node of $C$, its deformation parameter is an element of the characteristic monoid of $S$. 

\begin{definition}[The tropicalization of a log smooth curve]
Let $C\to S$ be a family of logarithmically smooth curves and assume that the underlying scheme of $S$ is a geometric point. The \textbf{tropicalization $C$} denoted $\plC$, is obtained as as follows: 

\begin{enumerate}[(A)]
\item the underlying graph is the marked dual graph of $C$ equipped with the standard genus and marking functions, 
\item given an edge $e$, the generalized length $\ell(e) = \delta_e\in \overline M_S$ is the deformation parameter of the corresponding node of $C$.
\end{enumerate}
\end{definition}

\subsection{Automorphisms} We use the following adjectives for curves \textit{without a chosen logarithmic structure}. A \textbf{prestable} curve is a possibly marked, nodal curve. A prestable curve is \textbf{semistable} if each rational component of the normalization contains at least $2$ distinguished points. A \textbf{stable curve} is one where each rational component of the normalization hosts at least $3$ distinguished points. 

We examine the infinitesimal automorphisms of logarithmically smooth curves. This is not original but only records the analysis of infinitesimal automorphisms of logarithmically smooth curves due to Santos-Parker in his thesis, advised by Wise~\cite[Section 3]{KSP-thesis}. The facts are implicit in our paper~\cite{RSW17A}. 

The subject of automorphisms of non-minimal logarithmic curves has not been expressly considered in the literature, and is more subtle than one would initially expect. We begin with a definition~\cite[Definition~B.3.1]{Che10}.

\begin{definition}
An \textbf{automorphism of a logarithmically smooth curve} $(C,M_C)\to (S,M_S)$ is a commutative diagram
\[
\begin{tikzcd}
(C,M_C)\arrow{d}\arrow{r} &(C,M_C) \arrow{d} \\
(S,M_S) \arrow{r} & (S,M_S),
\end{tikzcd}
\]
such that
\begin{enumerate}[(1)]
\item the horizontal arrows are isomorphisms of logarithmic schemes,
\item the underlying map $S\to S$ is the identity,
\item the underlying map $C\to C$ is an isomorphism.
\end{enumerate}
\end{definition}

Note that the logarithmic structure on the base is allowed to change.

\subsection{A first example} Let $(C,p)$ be a smooth rational curve equipped with its divisorial logarithmic structure. The natural $\mathbb G_a$ action that stabilizes $p$ yields automorphisms of the logarithmic structure. As a result, logarithmically smooth curves containing a rational component with only one distinguished point will always have infinitely many logarithmic automorphisms.

On the other hand, if $X$ is a semistable curve, the natural torus action coming from scaling the components does induce automorphisms of the logarithmic structure, leading to the next example.

\subsection{A second example} Let $C/S$ be a logarithmic curve over a geometric point, and let $C_1$ and $C_2$ be components of $C$ meeting at a node. In multiplicative notation, the equation of the logarithmic structure at this node is given by $xy = t$, where $t$ is a parameter on the base. Assume that $C_1$ is strictly semistable. We will analyze the circumstances in which the scaling action on $C_1$ may be extended to a logarithmic automorphism of $C/S$. 

The $\mathbb G_m$ action on $C_1$ induces a nontrivial action on the parameter $x$, by scaling it. In order to build a logarithmic automorphism $\varphi$, the new curve $\varphi(C)$ must also satisfy the equation $xy = t$. To achieve this, we must either scale $C_2$ inversely and fix $t$, or additionally scale $t$. 

Assume that $C_2$ is stable. This rules out a possible scaling action of the parameter $y$ to balance the action on $x$. The alternative is to scale $t$, which requires a nontrivial action on the base logarithmic structure. If the base logarithmic structure is minimal, then $t$ is not the smoothing parameter for any other node of $C$ besides the one between $C_1$ and $C_2$, and there is no obstacle to this scaling. We consider the alternative scenario: assume that there exists a third \textbf{stable} component $C_3$ meeting the stable component $C_2$ at a node with local equation $yz = t$. In this case, scaling $t$ does not result in an automorphism, since one of either $C_2$ and $C_3$ would also have to be scaled in order to produce an automorphism of $C/S$. It follows that the standard $\mathbb G_m$ action on $C_1$ does not extend to a logarithmic automorphism of $C/S$. 


\subsection{The general case} The example above illustrates a general situation. When the base logarithmic structure is not minimal the finiteness of automorphisms is not required for finiteness of logarithmic automorphism.

Recall that given a prestable curve $\underline C\to \underline S$ there is a \textbf{minimal} logarithmic structure obtained by pulling back the divisorial structures on the moduli stack of prestable curves and its universal curve. By identifying the stacks of prestable curves and prestable minimal logarithmic curves, we obtain the following. 

\begin{lemma}
Let $C/S$ be a minimal family of prestable logarithmic curves. Then there is an identification
\[
\mathrm{Aut}^{\mathrm{log}}(C/S) \cong \mathrm{Aut}(\underline{C/S}) 
\]
\end{lemma}

\begin{definition}
A logarithmic prestable curve $C/S$ is said to be \textbf{stable as a logarithmic curve} if it has finite logarithmic automorphism group. 
\end{definition}

To characterize such curves, we note the following.

\begin{definition}
Let $C/S$ be a logarithmically smooth curve and let $S$ be a geometric point. Given components $C_1$ and $C_2$ and a path $P$ between them, the \textbf{length} of $P$ is an element in the characteristic $\overline{M}_S$, given by the sum of deformation parameters corresponding to the edges in $P$. 
\end{definition}

A length $\delta\in \overline{M}_S$ is \textbf{strongly stable} if it is the length of a path $P$ between two stable components. 

\begin{definition}
Let $\delta\in \overline M_S$ be a section of the characteristic monoid. We say that $\delta$ is \textbf{stable} if it lies in the $\mathbb Q$-span of the lengths of strongly stable paths of $C$. That is, there exist paths $P_i$ with strongly stable lengths $\delta_i$ and rational numbers $q_i$ such that
\[
\delta = \sum_i q_i \delta_i \in \overline{M}_S\otimes \mathbb Q.
\]
\end{definition}

This notion of stability is stated for sections of $\overline M_S$, but certainly depends on $C$.

\begin{theorem}[{Santos-Parker~\cite[Proposition~6]{KSP-thesis}}]
Let $C/S$ be a logarithmically smooth curve over a geometric point. Let $\sigma_S$ be the dual of the monoid of the base, and let $\mathcal M^{\mathrm{trop}}$ be the cone of tropical curves of combinatorial type equal to that of $C$. The curve has trivial infinitesimal logarithmic automorphisms if and only if all of the following conditions are satisfied:
\begin{enumerate}[(A)]
\item the underlying curve is semistable,
\item the tropical moduli map $\sigma_S\to \mathcal M^{\trop}$ is injective,
\item every strict semistable rational component supports a node with smoothing parameter $\delta$ that is stable in $\overline M_S$.
\end{enumerate}

\end{theorem}

\begin{proof}
We refer the reader to the proof given by Santos-Parker in loc. cit., noting that the essential strategy of the proof is to globalize the observation in the example above. 
\end{proof}

\begin{proposition}
Let $C\to S$ be a stable logarithmic curve over a geometric point. Let $\widetilde{ C}\to  C$ be a logarithmic modification such that the composition $\widetilde C\to S$ is also a family of logarithmic curves. Then $\widetilde  C\to S$ is logarithmically stable.
\end{proposition}

\begin{proof}
We check the conditions of the characterization above. All such families arise by pulling back a subdivision on the tropicalization. It suffices to consider the subdivision at a single edge, producing $\widetilde\plC$ from $\plC$. The modified curve is clearly semistable since it is obtained by subdividing the tropical curve. Let $e$ be the subdivided edge, and assume first that $e$ is bounded (i.e. not a marked point). The edge $e$ has endpoints $v_1$ and $v_2$. Let $v_0$ be the newly formed vertex and let $\delta_1$ and $\delta_2$ be the deformation parameters of the edges incident to $v_0$. Since $v_1$ and $v_2$ are stable by hypothesis, the parameters $\delta_1$ and $\delta_2$ are stable parameters. Since $v_0$ supports these parameters, the curve remains stable. If $e$ is an unbounded edge, a similar argument applies with a single new bounded edge and a single new stable parameter supported on the new vertex. 
\end{proof}

Let $X$ be a logarithmic scheme. Given a logarithmic map from a curve, $f: C\to X$, \textbf{the logarithmic automorphisms of the map} are those automorphisms of $C$ that commute with the map. A \textbf{stable logarithmic map} is a logarithmic map whose automorphism group is finite. 

\begin{corollary}\label{cor: log-auts-maps}
Let 
\[
\begin{tikzcd}
C\arrow{r}\arrow{d} & X\\
S.&
\end{tikzcd}
\] 
be a stable logarithmic map over a geometric point. Let $\widetilde{C}\to  C$ be a logarithmic modification such that the composition $\widetilde C\to S$ is also a family of logarithmic curves. The resulting map $\widetilde C\to X$ is logarithmically stable.
\end{corollary}

\begin{remark}(Semistable curves)  Logarithmic curves with strictly semistable components, but finite automorphism group, have appeared in the literature. The moduli spaces of radially aligned curves constructed in~\cite{RSW17A} are examples, where every curve admits \textit{canonical} rational bubbling depending on the tropical structure. A related phenomenon appears in Smyth's genus $1$ curves with Gorenstein singularities~\cite[Corollary~2.4]{Smyth}. They are also a constant presence in Parker's exploded manifolds.
\end{remark}

\section{Combinatorial semistable reduction}\label{sec: comb-ss-reduction}

We construct tropical moduli stacks of maps that mimic the requirements of Theorem~\ref{thm: transverse-maps}. In this section, we will use $\Sigma$ to be a cone complex. It will play the role of the target in the moduli problem for maps, just as $X$ plays such a role on the geometric side. We occasionally use $\Sigma_X$ to denote the cone complex associated to $X$, in order to stress the connection to geometry. 


\subsection{Tropical maps: static target} Let $\Sigma$ be a rational cone complex in the sense of~\cite[Chapter 2,\S1]{KKMSD}. The reader may wish to assume that $\Sigma$ is a fan. Given a tropical curve $\plC$, a \textbf{tropical map} is a morphism of rational cone complexes 
\[
F: \plC\to \Sigma.
\]
In other words $F$ is continuous and maps each polyhedron (i.e. vertex or edge) of $\plC$ to exactly one polyhedron in $\Sigma$. Moreover, upon restriction to any edge $e$ of $\plC$, the map $F$ is piecewise linear. In particular, given an edge $e$ of $\plC$ it maps into a cone $\sigma_e$. If it is non-contracted, then its image is a segment, which lies on a line $L_e$ in the vector space $\sigma_e^{\mathrm{gp}}$. The induced map $e\to L_e$ has a well-defined slope up to a sign depending on the orientation of the edge. We refer to its absolute value as the \textit{expansion factor} of $F$ along $e$. 

If $\plC\to \Delta$ is a family of tropical curves over a cone complex $\Delta$, a family of tropical maps is a diagram of polyhedral complexes
\[
\begin{tikzcd}
\bplC\arrow{dr}\arrow{rr} & & \Sigma\times \Delta\arrow{dl} \\
& \Delta. &
\end{tikzcd}
\]
Let $e_p$ be a marked end of $\plC$. Let $\plC\to \Sigma$ be a tropical map with cone $\sigma_p$ containing $e_p$. 

\begin{definition}
The \textbf{contact order of $p$}, denoted $c_p$, is the scalar multiple of the primitive integral vector on the ray in $\sigma_p$ parallel to the image of $ e_p$ by the slope along $e_p$. If $e_p$ is contracted, declare its contact order to be zero. The contact order takes values in $\Sigma(\mathbb N)$. 
\end{definition}

\subsection{Tropical maps: moving target} We introduce the combinatorial target expansions. 

\begin{definition}\label{def: tropical-expansion}
A \textbf{tropical expansion} of $\Sigma$ over $\Delta$ is a cone complex $\widetilde \Sigma$ together with morphisms of cone complexes
\[
\begin{tikzcd}
\widetilde \Sigma \arrow{r}{\pi}\arrow{d} & \Sigma \\
\Delta &
\end{tikzcd}
\]
satisfying the following conditions
\begin{enumerate}[(1)]
\item {\bf Equidimensionality.} Every cone of $\widetilde \Sigma$ surjects onto a cone of $\Delta$.
\item {\bf Reducedness.} The image of the lattice of any cone $\sigma\in\widetilde \Sigma$ is equal to the lattice of the image cone $\pi(\sigma)$ in $\Delta$.
\item {\bf Modification.} For each point $p$ of $\Delta$ the map $\widetilde \Sigma_p\to \Sigma$ is an embedding of a subcomplex of a polyhedral subdivision of $\Sigma$.
\item {\bf Generic isomorphism.} The fiber over $0$ of $\widetilde \Sigma$, i.e. $\pi^{-1}(0)\subset \widetilde \Sigma$, maps isomorphically onto a union of faces in $\Sigma$. 
\end{enumerate}
\end{definition}

\noindent
The terms \textbf{equidimensionality} and \textbf{reducedness} reflect the following geometric situation. If $f:X\to B$ is a toroidal morphism of toroidal embeddings, the induced morphism $F$ of cone complexes satisfies the conditions above if and only if the map $f$ satisfies the corresponding algebro-geometric conditions, see~\cite[Lemmas~4.1 {\it \&}~5.2]{AK00}. If both are satisfied, the morphism is weakly semistable, and therefore flat~\cite[Theorem~2.1.5]{Mol16}.

\begin{remark}
If the cone complexes $\widetilde \Sigma$, $\Sigma$, and $\Delta$ above are fans embedded in a vector space, not necessarily complete, with linear maps between them, toric geometry outputs a family $ \mathcal Y(\widetilde \Sigma)\to S(\Delta)$. It is a flat family of broken toric varieties over the toric base $S(\Delta)$ associated to the fan $\Delta$. The generic fiber is the toric variety associated the preimage of $0\in\Delta$ of the map $\widetilde \Sigma\to \Delta$. By hypothesis, this is an open invariant subvariety of the toric variety $Y(\Sigma)$ associated to $\Sigma$. In general, since we do not demand that the fibers of a tropical expansion are complete decompositions, the fibers of the flat family need not be proper. 
\end{remark}

\begin{remark}
Given a point of $\Delta$, the fiber in $\widetilde \Sigma$ is a polyhedral subcomplex of a subdivision of $\Sigma$. The third condition allows us to take subdivisions and then discard higher dimensional cones if desired. The edge lengths of the bounded cells of this subdivision vary in a piecewise linear fashion based on the parameters in the base. An example of a family of polyhedral subdivisions over a one-dimensional cone is see in Figure~\ref{fig: subdivision-family}. 

\begin{figure}
\includegraphics[scale=0.35]{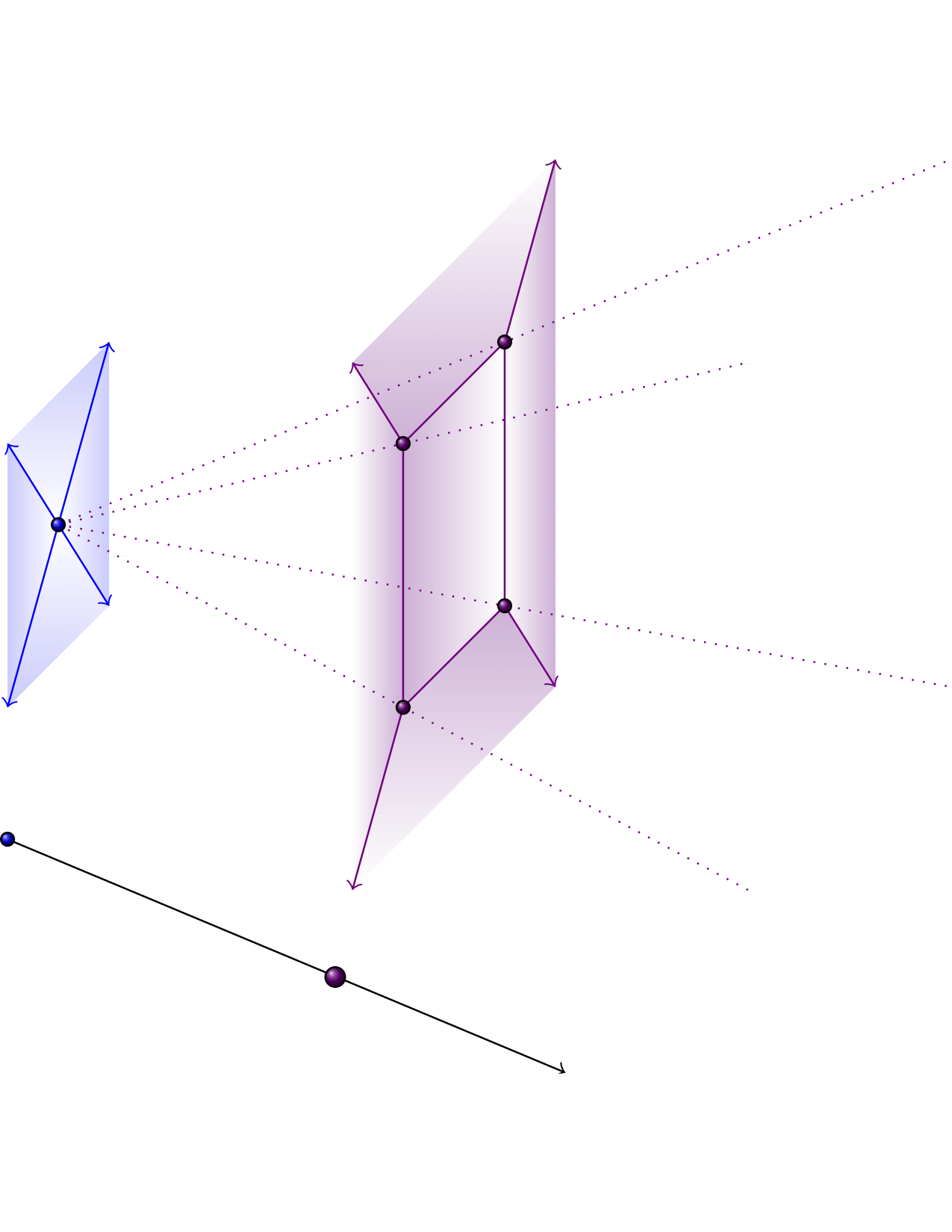}
\caption{A family of tropical expanded targets over $\RR_{\geq 0}$. The fiber over $0$ is dual to a family of $\PP^1\times\PP^1$'s, while the nontrivial fiber that is shown is dual to a union of four $\PP^2$'s glued along boundary curves, and all four $\PP^2$'s meet at a point. }\label{fig: subdivision-family}
\end{figure}

\end{remark}

\subsection{Conventions on combinatorial boundedness} A basic result in the logarithmic theory of stable maps is combinatorial finiteness. After fixing numerical invariants, there exist finitely many combinatorial types of tropical stable maps which can arise as combinatorial types of logarithmic stable maps to a fixed target~\cite[Section~3.1]{GS13}. When we speak of tropical moduli stacks, we fix a finite set of combinatorial types, closed under specialization. This guarantees that the tropical stacks are of finite type. The combinatorial type here is the data comprising of the stable dual graph of the source curve, the strata to which each component maps, the curve classes of the components, and the contact orders of edges and half edges corresponding to nodes and marked points. See~\cite{ACGS15}

\subsection{Tropical moduli stacks} Let $\Sigma$ be a cone complex. Fix the discrete data $\Gamma = (g,n,\bm{c},L)$, where $g$ is the genus, $n$ is the number of unbounded edges, $\bm{c}$ is the contact order of the markings, and $L$ is a finite partially ordered set of combinatorial types, closed under specialization of types. An \textbf{automorphism} of a map $\plC\to \Sigma$ is an automorphism of $\plC$ commuting with the map to $\Sigma$. 

We study moduli spaces of tropical maps to $\Sigma$ having combinatorial type described by $\Gamma$. The moduli problem determines a {combinatorial stack} in the formalism of Cavalieri--Chan--Ulirsch--Wise~\cite[Definition~2.3.5]{CCUW}. 

Let $\mathbf{RPC}$ be the category of rational polyhedral cones. The definitions give rise to a tropical moduli functor
\[
\mathsf{ACGS}_\Gamma(\Sigma): \mathbf{RPC}^{\mathrm{op}}\to \mathbf{Groupoids},
\]
by associating to a cone $\tau$ the groupoid of tropical prestable maps $\plC\to \Sigma$ over $\tau$ with discrete data $\Gamma$. The following is a modest adaptation of the results in loc. cit.

\begin{lemma}[{\cite[Adaptation of Theorem~1]{CCUW}}]
There is a stack $\mathsf{ACGS}_\Gamma(\Sigma)$ over the category of cone complexes representing the fibered category over cones of tropical stable maps to $\Sigma$ with discrete data $\Gamma$. 
\end{lemma}

\subsection{Combinatorial transversality} A basic transversality statement motivates what follows. Let $X$ be a toroidal embedding and $C\to X$ a logarithmic stable map. Consider its tropicalization
\[
\plC\to \Sigma_X.
\]
We seek a tropical condition for the underlying map $C\to X$ to be {transverse} in the following sense: (i) the complement of the marked points $C^\circ\subset C$ maps into the open stratum of $X$, and the marked points map to the complement of the closed codimension $2$ strata in $X$. This holds precisely when the image of every vertex of $\plC$ is the unique vertex $0$ of $\Sigma_X$, and the image of every unbounded edge is a ray of $\Sigma_X$. For a more general map, vertices and edges of $\plC$ may map into high dimensional cones of $\Sigma_X$. Geometrically, these are situations where components and nodes of $C$ map into high codimension boundary strata of $X$.

Toroidal blowups of $X$, which are \textit{conical} subdivisions of $\Sigma_X$, do not improve the situation. However, given a general map $\plC\to \Sigma_X$, if we are allowed to make \textit{polyhedral} subdivisions that are not necessarily conical, the image of every vertex can be arranged to be a vertex, and the image of every edge an edge. We refer to this as \textbf{combinatorial transversality}. Such polyhedral subdivisions determine expansions of the target. A read who is unfamiliar with such expansions may wish to consult~\cite[Section~3]{NS06} or~\cite[Sections 6 and 7]{Gub13}. A discussion of subdivisions may be found in~\cite{AW,Kato94,KKMSD}. 

\subsection{The construction on standard points} The main construction is a more delicate version of the following simple ``pointwise'' construction. Given a single tropical stable map $\plC\to \Sigma$, viewing $\plC$ as a metric graph, one can choose a polyhedral -- not necessarily conical -- subdivision $\widetilde\Sigma$ such that the image of $\plC$ is the support of a polyhedral subcomplex, see Figure~\ref{fig: subdivision}. Once this subdivision is made, the map $\plC\to \widetilde \Sigma$ may not be polyhedral. To remedy this, a subdivision to $\plC$ can be made by adding vertices to $\plC$ at the preimages of the vertices of $\widetilde\Sigma$. 

\begin{figure}[h!]
\includegraphics[scale=0.4]{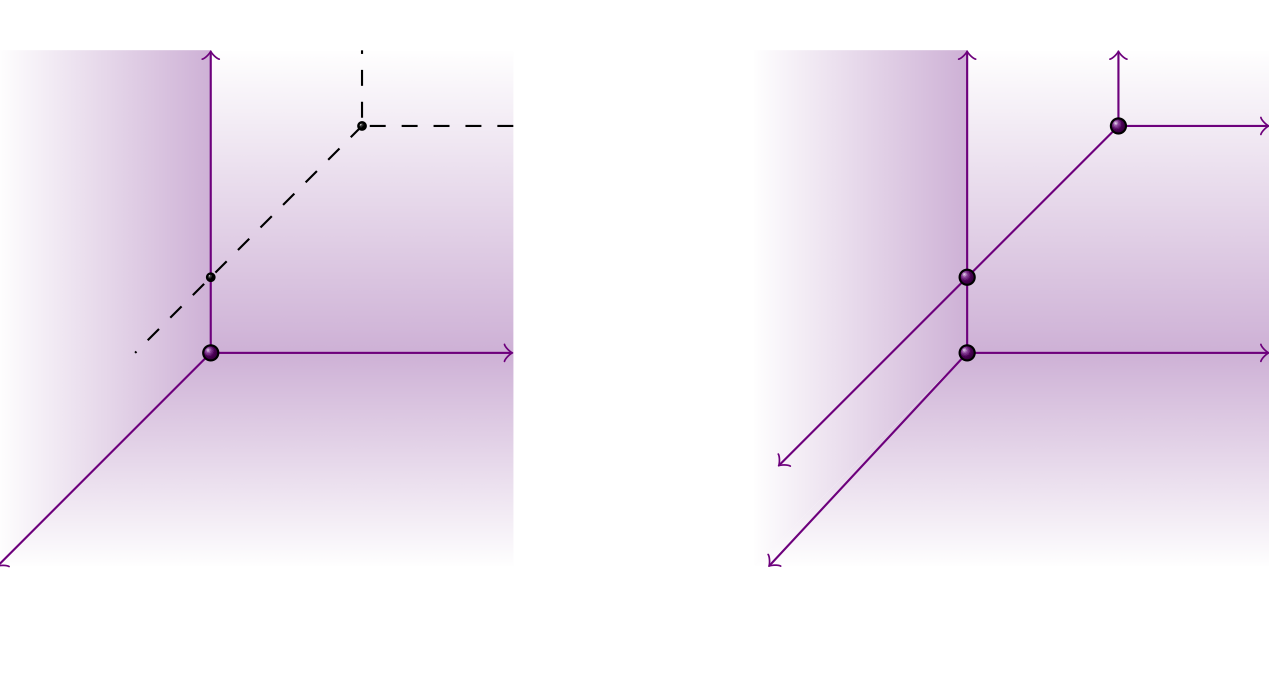}
\caption{The dashed figure on the left is a tropical curve in the fan $\Delta_{\PP^2}$. It is clearly not combinatorially transverse because the vertices and edges of the tropical curve do not map onto polyhedral subcomplexes of the target. There is a natural subdivision of this on the right, obtained by introducing new vertices and edges. The tropical curve on the right is not drawn on the right, but it is the one obtained by composition with the subdivision.}
\label{fig: subdivision}
\end{figure}

\begin{remark}
There always exists a subdivision of the target to make a map transverse. Fix a tropical map $\plC\to \Sigma$. Replace $\Sigma$ with a conical subdivision so that it becomes simplicial. Since the question is local, it suffices to deal with the case where $\Sigma$ is $\RR_{\geq 0}^n$. For each of the $n$ projections to $\RR_{\geq 0}$, choose a polyhedral subdivision $\mathsf{T_i}$ such that the image of every vertex is a vertex, and the image of every noncontracted bounded edge is a union of bounded edges. The product of the $\mathsf T_i$ yields a cubical subdivision of $\RR_{\geq 0}^n$ where the image of every vertex is a vertex. The image of an edge is the diagonal in a facet connecting two edges. By further triangulating, we guarantee that the image of each edge is an edge in the subdivision, as required. By passing to the $1$-skeleton of this subdivision, and if necessary discarding unnecessary vertices and edges, we obtain a non-complete polyhedral subdivision of $\Sigma$. 
\end{remark}

Geometrically, this procedure expands the target to accommodate a given one-parameter family of curves with smooth general fiber. The same is true if the base is a finite rank valuation ring, or more generally a valuativized logarithmic scheme~\cite{KatoLDDT}. In these cases, the new families are guaranteed to be flat. Over a general base this ceases to be true, and the base must be modified to ensure flatness.

\subsection{The general case}\label{sec: transverse-maps} We establish the tropical version of Theorem~\ref{thm: transverse-maps} by globalizing the construction above. The tropical moduli stack comes with a curve and map diagram: 
\begin{equation}\label{eqn: universal-tropical}
\begin{tikzcd}
\bplC \arrow{rr}{F} \arrow{dr} & & \mathsf{ACGS}(\Sigma)\times\Sigma \arrow{dl}\\
& \mathsf{ACGS}(\Sigma). &
\end{tikzcd}
\end{equation}
The construction will subdivide the base, curve, and target to obtain a universal family of stable maps to a family of expansions of $\Sigma$, such that the universal map is transverse to the universal target.

\begin{theorem}\label{thm: main-comb-thm}
There exists a diagram of cone stacks
\[
\begin{tikzcd}
\bplC^\lambda \arrow{rr}{F^\lambda} \arrow{dr} & & \widetilde\Sigma^\lambda \arrow{dl}\\
& \mathsf{K}^\lambda_\Gamma(\Sigma). &
\end{tikzcd}
\]
where each position in the diagram is a subdivision of the corresponding position in Equation~(\ref{eqn: universal-tropical}). Furthermore,
\begin{enumerate}[(1)]
 \item {\bf Equidimensionality.} For both vertical arrows, every cone of the source surjects onto a cone of $\mathsf{K}^\lambda_\Gamma(\Sigma)$.
\item {\bf Reducedness.} For both vertical arrows, the image of the lattice of any cone $\sigma$ is equal to the lattice in the image cone in $\mathsf{K}^\lambda_\Gamma(\Sigma)$.
\item {\bf Transversality.} The image of $F^\lambda$ is a union of faces of $\widetilde\Sigma^\lambda$.
\end{enumerate}
\end{theorem}

\begin{proof}
The basic idea of the proof is to perform weak semistable reduction for the universal map at a combinatorial level, following~\cite[Sections 4 {\it \&} 5]{AK00}. We guidepost the proof based on the intended geometric property that each step reflects. 

\noindent
{\bf I: Target expansion.} The morphism 
\[
F: \bplC \to \mathsf{ACGS}(\Sigma)\times\Sigma
\] 
is linear on each cone of the source. Given a cone $\tau$ of $\bplC$, the image $F(\tau)$ is a conical subset of $\mathsf K_\Gamma(\Sigma)\times\Sigma$. Choose a subdivision of $\widetilde\Sigma^{\mathrm{aux}}$ of $\mathsf{ACGS}(\Sigma)\times\Sigma$ such that the image of every cone of $\bplC$ is a union of cones. This can be done using a procedure developed by Abramovich and Karu, see~\cite[Lemma~4.3]{AK00}. Assume first that $ \mathsf{ACGS}(\Sigma)\times\Sigma$ consists of a single maximal cone $\sigma$. Then the image $F(\tau)$ is a strictly convex conical subset of $\sigma$. This subset is equal to the intersection of half-spaces given by linear functions $\ell_{\tau,j}$. The function $\psi_\sigma = -\sum_{\tau,j}|\ell_{\tau,j}|$ is piecewise linear, and the domains of linearity define a subdivision of $\sigma$ such that $\tau$ maps to a cone. 

Once a subdivision of $\sigma$ is made such that the image of a cone is a union of cones, a further refinement does not affect the property. Repeat this procedure for all cones mapping to $\sigma$, thus defining a subdivision where the image of every cone is a \textit{union} of cones. Now choose an arbitrary piecewise linear function $\overline \psi_\sigma$ on the full target $\mathsf{ACGS}(\Sigma)\times\Sigma$ extending the function $\psi_\sigma$ constructed above. Consider the piecewise linear function
\[
\psi = \sum_{\sigma} \overline \psi_\sigma.
\]
The domains of linearity of $\psi$ give the requisite subdivision such that the image of every cone of the curve is a union of cones of the target. Call this subdivision of the target $\widetilde\Sigma^{\mathrm{aux}}$

\noindent
{\bf II: Transversality.} Pull back the piecewise linear function $\psi$ to $\bplC$. Its bending locus yields a subdivision of the source curve, and the resulting map sends cones onto cones. Pass to the subcomplex $\widetilde\Sigma^{\mathrm{aux}}_\circ$ of $\widetilde\Sigma^{\mathrm{aux}}$ that contains the images of the cones of $\bplC^{\mathrm{aux}}$. The newly constructed morphism of cone complexes
\[
F^{\mathrm{aux}}: \bplC^{\mathrm{aux}}\to \widetilde\Sigma^{\mathrm{aux}}_\circ
\]
has the property that the image of every cone of $\bplC^{\mathrm{aux}}$ is a cone of $\widetilde\Sigma^{\mathrm{aux}}_\circ$, and the map is surjective.

\noindent
{\bf III: Fixing the base.} We have thus far ignored the base. The modification $\widetilde\Sigma^{\mathrm{aux}}_\circ$ maps surjectively onto $\Kbar_\Gamma(\Sigma)$. Indeed, the map factors through the product $\mathsf{ACGS}_\Gamma(\Sigma)\times\Sigma$. Another application of~\cite[Lemma~4.3 {\it\&} Proposition 4.4]{AK00} yields subdivisions of the target and base
\[
\widetilde\Sigma^{\lambda}_\circ\to \Kbar^\lambda_\Gamma(\Sigma)
\]
such that the image of every cone of the source is a cone of the target. Pull back this subdivision of $\widetilde\Sigma^{\lambda}_\circ\to \widetilde\Sigma^{\mathrm{aux}}_\circ$ to the curve to obtain $\bplC^\lambda$. The image of any cone of $\bplC^\lambda$ is a cone in $ \mathsf{K}^\lambda_\Gamma(\Sigma)$ by commutativity of the diagram. Finally, by refining the lattice on the base we guarantee that the reducedness condition in the statement.
\end{proof}

In the transversality step, we passed from $\widetilde \Sigma_\circ^{}$ to the subcomplex $\widetilde \Sigma_\circ^{\mathrm{aux}}$. The step can be skipped without significant consequences, but its effect here is that we produce a non-proper target expansion that nonetheless contains the image of the curve. 

\begin{remark}
The proof above makes no reference to the economics of the chosen subdivisions. For computational purposes, particular choices may be advantageous, A systematically understanding is likely to prove important in later work. 
\end{remark}

\subsection{Smooth divisor: $\mathsf{Li}$ from $\mathsf{ACGS}$} We specialize the construction to the case considered traditionally, where $D$ is a smooth divisor. Combinatorially, the fan $\Sigma$ is $\RR_{\geq 0}$. At each point of $\mathsf{ACGS}(\RR_{\geq 0})$ there is a piecewise linear map
\[
F: \plC\to\RR_{\geq 0},
\]
For each vertex $v$ in $\plC$, its image $F(v)$ is generally not the vertex of $\RR_{\geq 0}$. The images of all vertices of $\plC$ form the vertices in a unique polyhedral subdivision $\mathsf{T}$ of $\RR_{\geq 0}$. Geometrically, this canonically determines an expansion of the target in the traditional sense, by pulling back the subdivision along the tropicalization. However, the map $\plC \to \mathsf{T}$ no longer maps cones to cones, see Figure~\ref{fig: target-modification}. To remedy this, the subdivision may be pulled back to $\plC$ to obtain a subdivision of the source curve. 

\begin{figure}[h!]
\includegraphics[scale=0.175]{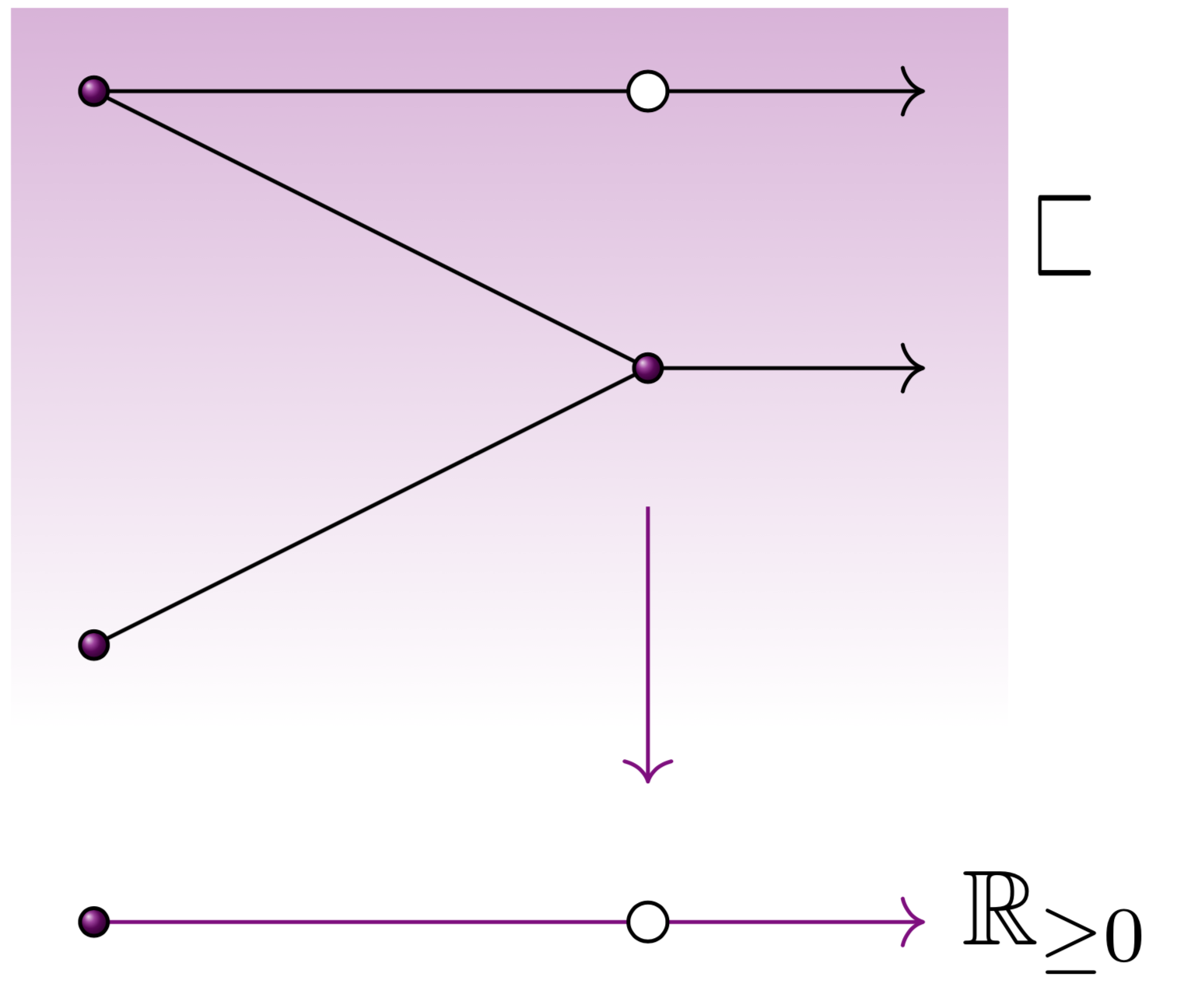}
\caption{Ignoring the white vertices, the map in the figure above is a tropical map arising from points in the tropicalization $\mathsf{ACGS}$ space. Subdivision of the target is the expansion arising in Jun Li's space. The subdivision of the source is made by marking the preimages of the newly introduced white vertex.}\label{fig: target-modification}
\end{figure}

The curve and target families are no longer equidimensional over the base, because the newly introduced vertices in the curve (or the target) form a cone in the total space of the curve (or the target) over $\mathsf{ACGS}(\RR_{\geq 0})$. The image of this cone is not necessarily a cone in the base. The subdivision of the base produces the subdivision $\mathsf{Li}\to\mathsf{ACGS}$. This combinatorics is analyzed in some detail in~\cite[Proposition~16, Example~17]{CMR14b}. See also~\cite[Section~6]{GS13}. 

\begin{remark}
With hindsight, the paper of Chen on logarithmic stable maps may be viewed in different light, as reversing the construction above~\cite{Che10}. Assume the existence of the moduli space $\mathsf{Kim}$ of logarithmic maps to expansions of the universal target $[\mathbb A^1/\mathbb G_m]$, as in~\cite{Kim08}. Then we may blow down the universal expanded target family to the constant family $[\mathbb A^1/\mathbb G_m]$ and partially stabilize the source curve, over $\mathsf{Kim}$. Observe that the universal map continues to be logarithmic, with the pushforward logarithmic structure~\cite[\textit{Appendix~B}]{AMW12}. This gives a new logarithmic source and target family over $\mathsf{Kim}$. We observe that combinatorial types in $\mathsf{Kim}$ come with a partial ordering, where two otherwise identical combinatorial types are distinct, if the images of two vertices in $\RR_{\geq 0}$ have different orders. By forgetting this partial ordering we obtain a collapsing of some faces in $\mathsf{Kim}$, and consequently a blow down map $\mathsf{Kim}\to\mathsf{Chen}$, where the universal curve continues to be flat. This outlines why the space of logarithmic maps to the universal target, without expansions, is representable by a stack. Of course, it is unlikely that the spaces would have been discovered in this way to begin with.
\end{remark}

\subsection{Auxiliary combinatorial choices} The procedure in the proof of Theorem~\ref{thm: main-comb-thm} produces infinitely many diagrams satisfying the requirements of the theorem. Indeed, one can perform an arbitrary further refinement of the base moduli space and pull back the source and target families to obtain a new moduli space satisfying the same conditions. Given any two such solutions to this moduli problem, there is a common subdivision of all parts of the diagram. 

\begin{lemma}\label{common-subdivision}
Let $\Kbar^{1}_\Gamma(\Sigma)$ and $\Kbar^{2}_\Gamma(\Sigma)$ be two diagrams, equipped with universal families, arising in the result above. Then there exists a third diagram $(\Kbar^\nu_\Gamma(\Sigma),\bplC^\nu,\widetilde\Sigma^\nu)$ together with subdivisions
\[
\mathrm{sub_i}: \Kbar^\nu_\Gamma(\Sigma)\to \Kbar^{i}_\Gamma(\Sigma),
\]
that are natural subdivisions of the pullback 
\[
\bplC^\nu\to \mathrm{sub_i}^{-1}\bplC^{i}
\]
and 
\[
\widetilde\Sigma^\nu\to \mathrm{sub_i}^{-1}\widetilde \Sigma^{i}.
\]
\end{lemma}

\begin{proof}
The existence of $\mathrm{sub_i}: \Kbar^\nu_\Gamma(\Sigma)\to \Kbar^{i}_\Gamma(\Sigma)$ is straightforward. After pulling back the universal tropical curve and target from the two spaces, further subdivisions can be made. The equidimensionality procedure described in the proof of the previous theorem constructs the requisite diagram.
\end{proof}

\section{Virtual semistable reduction}\label{sec: v-ss-red}

Recall that $(X,D)$ is a toroidal pair without self intersection. In this section, we establish Theorem~\ref{thm: transverse-maps}. The moduli spaces $\Kbar^\lambda_\Gamma(\cdot)$ are first constructed when the target is the Artin fan of $X$, and then for $X$ itself in Section~\ref{sec: maps-from-subs}. The transversality is established in Lemma~\ref{thm: transversality}. The virtual class is discussed in Section~\ref{sec: vir-class}.

The procedure in this section has been used before, for instance in~\cite{AW,CMR14b,R15b}. When the moduli space of maps is logarithmically unobstructed, for instance when considering genus $0$ logarithmic maps to toric varieties, subdivisions of the tropical moduli spaces produce toroidal modifications of the geometric moduli spaces. If $(X,D)$ is replaced with its Artin fan $\mathcal A_X$, the moduli space of maps becomes a toroidal embedding in the smooth topology. This toroidal embedding maps to an associated cone stack. Subdivisions from the previous section are applied to this cone stack, and then pulled back to maps $X$ itself. We now formalize the idea.

\subsection{Artin fans} A toroidal embedding without self-intersection $(X,D)$ determines an associated cone complex~\cite{KKMSD}. At every point of $X$, there is a Zariski open set $U\subset X$ together with an \'etale map $U\to V_{\sigma(p)}$ where $V_{\sigma(p)}$ is an affine toric variety with cone $\sigma(p)$. The association $p\mapsto\sigma(p)$ is constant on the locally closed toroidal strata of $X$. By considering the specialization relations among generic points of strata, these can be glued to form a cone complex $\Sigma_X$. When the boundary of $X$ is simple normal crossings, this is the cone over the dual complex of the boundary divisor $D$.

The cone complex $\Sigma_X$ determines a $0$-dimensional algebraic stack possessing the same combinatorial properties as $X$, referred to as the {\bf Artin fan of $X$}. For a general construction, we refer the reader to~\cite{ACMW,AW,CCUW,U16} and their precusor~\cite{Ols03}. In our case, the following will suffice. 

The cone complex $\Sigma_X$ above is a colimit of a diagram $\mathscr D_X$ of cones~\cite[Remark~2.2.1]{ACP}. For each cone $\sigma$ in the diagram, there is a canonically associated toric stack $[V_\sigma/T]$, where $T$ is the dense torus of $V_\sigma$. The diagram $\mathscr D_X$ immediately yields a diagram of algebraic stacks, whose colimit exists as an algebraic stack called the \textbf{Artin fan}. There is a canonical morphism
\[
X\to \mathcal A_X.
\]
This morphism is strict: the logarithmic structure on $X$ is the pullback of the divisorial logarithmic structure on $\mathcal A_X$. The category of cone complexes is equivalent to the category of Artin fans without monodromy~\cite[Proposition~3.3]{U16}. 

\subsection{Toroidal modifications} Given a logarithmic scheme or stack, a {\bf logarithmic modification}, sometimes called a toroidal modification, is a proper, birational, logarithmically \'etale morphism $X'\to X$. If $X$ is a toroidal embedding, such modifications $X'\to X$ correspond in a functorial manner to polyhedral refinements $\Sigma'\to \Sigma$ of the corresponding cone complexes, and thus to logarithmic modifications of the associated Artin fans. See~\cite[Section~3.1]{ACMW}.

\subsection{Maps to the Artin fan} The primary purpose of working with Artin fans is that it is natural to map algebraic varieties to them. 

\begin{definition}
A \textbf{logarithmic prestable map to $\mathcal A_X$ over $S$} is a diagram of logarithmic stacks
\[
\begin{tikzcd}
\mathcal C\arrow{r}\arrow{d} & \mathcal A_X\\
S,&
\end{tikzcd}
\]
where $\mathcal C$ is a family of logarithmically smooth curves.
\end{definition}


Fix discrete data $\Gamma$ compirising the genus $g$, number of marked points $n$, and contact orders $\bm c$. The logarithmic tangent bundle of $\mathcal A_X$ is the zero sheaf, so logarithmic deformations of maps from curves are unobstructed over the moduli space of curves. This leads to the following result of Abramovich and Wise~\cite[Proposition~1.6.1]{AW}.

\begin{proposition}
There exists a logarithmic algebraic stack $\fM_\Gamma(\mathcal A_X)$, which is toroidal in the smooth topology, of dimension $3g-3+n$, and represents the fibered category over logarithmic schemes of logarithmic maps to $\mathcal A_X$. The forgetful morphism to $\fM_{g,n}$ is logarithmically \'etale and birational.
\end{proposition}

As constructed, the stack is only locally of finite type. Indeed, a logarithmic map $C\to \mathcal A_X$ over a logarithmic base $S$ is equivalent to the data of $C$, and a map $\plC\to \Sigma_X$. Without a balancing condition on the maps, the problem is unbounded. However, the space of \textit{stable} logarithmic maps $\mathsf{ACGS}_\Gamma(X)$ is of finite type. We simply replace $\fM_\Gamma(\mathcal A_X)$ with any open substack that is of finite type and receives a morphism from $\mathsf{ACGS}_\Gamma(X)$. 

\begin{remark}
If $\Sigma$ is a fan embedded in a vector space, the map $\plC\to \Sigma$ can be required to be balanced. That is, we demand that the sum of the directional derivatives at every point of $\plC$ is $0$. It has been shown that there are only finitely many combinatorial types, and this can be seen via a Newton polygon argument~\cite[\textit{Proposition~2.1}]{NS06}. In the case considered here, $\Sigma$ can be embedded in a vector space, but tropicalizations of maps $C\to X$ need not be balanced. However, a modified balancing condition does hold, and this is used to establish boundedness in~\cite[\textit{Theorem~3.8}]{GS13}. Boundedness holds in greater generality~\cite[\textit{Section~1.1}]{ACMW}. 
\end{remark}

Let $\mathsf{ACGS}_\Gamma(\mathcal A_X)$ be the resulting finite type stack of logarithmic maps to $\mathcal A_X$. By enumerating the combinatorial types that are relevant to $X$, which we know are finite by boundedness, we obtain the analogous tropical moduli stack $\mathsf{ACGS}_\Gamma(\Sigma)$, where $\Sigma$ is the cone complex of $X$. The moduli problem $\mathsf{ACGS}_\Gamma(\Sigma)$ on cones from the previous section can be turned in to a moduli problem on logarithmic schemes, as follows. Given a logarithmic scheme $(S,M_S)$ that is locally of finite type, we consider the moduli problem of tropical stable maps with discrete data $\Gamma$ over $S^{\trop}$. This extends the tropical moduli functor to logarithmic schemes, see~\cite[Part II]{CCUW}. 

By an insight of Gross and Siebert~\cite[Section~1]{GS13}, a family of logarithmic maps to $\mathcal A_X$ over $S$ gives rise to a family of tropical stable maps to $\Sigma$ over $S^{\trop}$. Since $\Kbar_\Gamma(\Sigma)$ is universal for this moduli problem, we obtain a morphism
\[
\trop: \mathsf{ACGS}_\Gamma(\mathcal A_X)\to \mathsf{ACGS}_\Gamma(\Sigma).
\]
The notion of a morphism from the logarithmic algebraic stack on the left to the cone stack on the right is explained in detail in~\cite{CCUW}. Briefly, the authors prove that cone stacks form an equivalent $2$-category to Artin fans, and the tropicalization map is constructed by using this categorical equivalence. In that context, the following proposition is essentially identical to~\cite[Theorem~4]{CCUW}.

\begin{proposition}
The morphism 
\[
\trop: \mathsf{ACGS}_\Gamma(\mathcal A_X)\to \mathsf{ACGS}_\Gamma(\Sigma).
\]
is surjective, smooth and strict.
\end{proposition}

\begin{proof}
The morphism is surjective by construction. By appealing to~\cite[Section~1]{GS13}, the minimal monoids of $\mathsf{ACGS}_\Gamma(\mathcal A_X)$ are identified with the monoids dual to the tropical moduli space, so strictness is a tautology. Smoothness follows from the logarithmic smoothness of $\mathsf{ACGS}_\Gamma(\mathcal A_X)$.
\end{proof}

\subsection{Spaces of maps from the subdivisions}\label{sec: maps-from-subs} With the tropicalization map handy, we pull back the subdivisions constructed in Section~\ref{sec: transverse-maps}. These spaces will satisfy the transversality hypothesis in Theorem~\ref{thm: transverse-maps}.

Let $X$ be a toroidal embedding without self intersections and let $\Sigma$ be its cone complex. Let $S$ be a logarithmic scheme with tropicalization $\Delta$. 

\begin{definition}
Given a tropical expansion $\widetilde \Sigma$  of $\Sigma$ over $\Delta$ in the sense of Definition~\ref{def: tropical-expansion}, the associated \textbf{logarithmic expansion} is obtained as the logarithmic modification
\[
\widetilde{\mathcal X} := (X \times S)\times_{\Sigma\times \Delta} \widetilde \Sigma \to S.
\]
\end{definition}

The reader might be perturbed that the definition above has a logarithmic scheme on one side and a cone complex on the other -- this is consistent with the formalism introduced in~\cite{CCUW}. In practice, we pass to the associated Artin fan on the cone complex side and take the fiber product in fine and saturated logarithmic schemes. 

As discussed previously, the family $\widetilde{\mathcal X}\to S$ has reduced and equidimensional fibers by construction. In fact, it is also flat~\cite[Theorem~2.1.5]{Mol16}. We construct moduli spaces of curves mapping to expansions of $X$.

Fix a tropical moduli diagram 
\[
\begin{tikzcd}
\bplC^\lambda \arrow{rr}{F^\lambda} \arrow{dr} & & \widetilde\Sigma^\lambda \arrow{dl}\\
& \Kbar^\lambda_\Gamma(\Sigma). &
\end{tikzcd}
\]
as constructed in Theorem~\ref{thm: main-comb-thm}. The subdivision $\Kbar^\lambda_\Gamma(\Sigma)$ is a subcategory of the fibered category $\mathsf{ACGS}_\Gamma(\Sigma)$. By pulling back $\Kbar_\Gamma(\mathcal A_X)$ via the tropicalization map we obtain a subcategory $\Kbar^\lambda_\Gamma(\mathcal A_X)$. We similarly obtain a universal curve family $\mathcal C^\lambda$ and a target family $\mathcal A_X^\lambda$. 

\begin{lemma}[{\it Representability}]\label{prop: representability}
The categories $\Kbar^\lambda_\Gamma(\mathcal A_X)$, $\mathcal C^\lambda$ and $\mathcal A_X^\lambda$ are representable by algebraic stacks with logarithmic structure. 
\end{lemma}

\begin{proof}
Each of the categories is obtained locally by pulling back toric modifications, and possibly root constructions, of toric varieties, so representability is immediate. Note that the root constructions have the effect of adding finite isotropy groups to the strata, and correspond to the passage to a finite index sublattice at the level of cone stacks. The root constructions are required to ensure the the curve $\mathcal C^\lambda$ is reduced over the moduli space. See for instance~\cite[Section~3]{Mol16} for a discussion.
\end{proof}

\begin{proposition}
The morphism $\mathcal C^\lambda\to \mathsf{K}^\lambda_\Gamma(\mathcal A_X)$ is a flat family of logarithmically smooth curves. The Artin fan target family $\mathcal A_X^\lambda\to \mathsf{K}^\lambda_\Gamma(\mathcal A_X)$ is also flat.
\end{proposition}

\begin{proof}
Given the modification $\Kbar^\lambda_\Gamma(\mathcal A_X)\to \mathsf{ACGS}(\mathcal A_X)$, we may pull back the universal curve to obtain 
\[
\mathcal C\to \Kbar^\lambda_\Gamma(\mathcal A_X).
\]
This family is certainly flat since it is obtained by pulling back a flat family. The new universal curve is a modification
\[
\mathcal C^\lambda\to\mathcal C,
\]
and we wish to show that the composition $\mathcal C^\lambda\to \Kbar^\lambda_\Gamma(\mathcal A_X)$ remains flat. This can be checked using the polyhedral criteria~\cite[Sections~4 {\it \&} 5]{AK00}, which are satisfied by construction. The same holds for the target family.
\end{proof}

For properties relating to characteristic sheaves, $X$ and $\mathcal A_X$ are indistinguishable~\cite[Lemma~4.1]{AW}.

\begin{lemma}
The morphism
\[
\mathsf{ACGS}_\Gamma(X)\to \mathsf{ACGS}_\Gamma(\mathcal A_X)
\]
is strict.
\end{lemma}

We obtain the analogous variants of actual stable maps to expansions of $X$, denoted $\mathsf{K}^\lambda_\Gamma(X)$.

\begin{corollary}
The curve and target families $\mathcal C^\lambda\to \mathsf{K}^\lambda_\Gamma(X)$ and $\mathcal X^\lambda\to \mathsf{K}^\lambda_\Gamma(X)$ are flat. 
\end{corollary}

\begin{proof}
The curve is pulled back from the space of maps to the Artin fan, so its flatness is clear. Similarly, over $\Kbar^\lambda_\Gamma(\mathcal A_X)$, there is a modification
\[
\mathcal A_X^\lambda\to \mathcal A_X\times \Kbar^\lambda_\Gamma(\mathcal A_X).
\]
The family $\mathcal A_X^\lambda$ is flat over $\Kbar^\lambda_\Gamma(\mathcal A_X)$ by construction. We also have the constant target family $X\times \Kbar^\lambda_\Gamma(\mathcal A_X)\to \Kbar^\lambda_\Gamma(\mathcal A_X)$. The modification yields $\mathcal X^\lambda\to \Kbar^\lambda_\Gamma(\mathcal A_X)$, which is the composition of a strict logarithmically smooth morphism $\mathcal X^\lambda\to \mathcal A_X^\lambda$ and a flat morphism, and therefore remains so. The target family on $\Kbar^\lambda_\Gamma(X)$ is pulled back from the space of maps to the Artin fan, so it is flat. 
\end{proof}

\begin{lemma}[{\it Properness}]
The fibered category $\Kbar^\lambda_\Gamma(X)$ is representable by a proper Deligne--Mumford stack with logarithmic structure.
\end{lemma}

\begin{proof}
The analogous stack $\Kbar^\lambda_\Gamma(\mathcal A_X)$ is algebraic by Lemma~\ref{prop: representability}. The morphism
\[
\Kbar^\lambda_\Gamma(\mathcal A_X)\to\mathsf{ACGS}_\Gamma(\mathcal A_X)
\]
is proper, since it is pulled back from a subdivision. Therefore, $\Kbar^\lambda_\Gamma(X)$ is a proper modification of the proper Deligne--Mumford stack $\mathsf{ACGS}_\Gamma(X)$, and therefore is representable and proper.
\end{proof}

The subdivision may add arbitrary numbers of bubbles to the curve that map to an expansion of the target. Logarithmic stability is unaffected. 

\begin{lemma}[{\it Logarithmic stability}]
Given a point of the logarithmic algebraic stack $\Kbar^\lambda_\Gamma(X)$, the corresponding map
\[
C^\lambda \to \mathcal X^\lambda\to X
\]
has finite logarithmic automorphism group.
\end{lemma}

\begin{proof}
This follows immediately from Corollary~\ref{cor: log-auts-maps}.
\end{proof}

The stacks inherit the requisite transversality. Let $C^\lambda\to \mathcal A_X^\lambda\to \mathcal A_X$ be a point of $\Kbar^\lambda_\Gamma(\mathcal A_X)$. 

\begin{lemma}[{\it Transversality}]\label{thm: transversality}
Let $p_i$ be a marked point of $C^\lambda$ with contact order $c_{ij}$ along the divisor $D_j\subset \mathcal A_X$ and zero contact order with the remaining divisors. Let $D_j^\lambda$ be the stratum of the target $\mathcal X^\lambda$ containing the image of $p_i$. Then $C^\lambda$ meets $D_j$ at finitely many points, each marked, and the order of tangency at $p_i$ is equal to $c_{ij}$. 
\end{lemma}

\begin{proof}
Consider a geometric point of the algebraic stack $\Kbar^\lambda_\Gamma(\mathcal A_X)$. If the stalk of the characteristic sheaf at this point is a monoid $P$, then by pullback, we obtain a a logarithmic map to $\mathcal A_X$ over $\spec(P\to \CC)$. The dual cone $\sigma_P$ is a cone of $\Kbar^\lambda_\Gamma(\Sigma)$. Let $m$ denote the dimension of the cone $\sigma_P$. We examine the associated family of tropical maps $F:\plC^\lambda\to \Sigma^\lambda\to \Sigma$. Let $G$ be the underlying graph of $\plC$. Every vertex $v$ of $G$ determines an $m$-dimensional cone of the cone complex $\plC$. Note that this cone has relative dimension $0$ over $\Kbar^\lambda_\Gamma(\Sigma)$. Its image in $\Sigma$ is also an $m$-dimensional cone, and therefore the generic point of the component $C_v$ dual to $v$ maps to the generic point of the component of $\mathcal A_X^\lambda$ dual to $F(v)$. Similarly, the edges of $\plC$ which are dual to nodes and marked points, map to either $m$ or $m+1$ dimensional cones of $\Sigma^\lambda$, which are dual to either components or double intersections of components $\mathcal A^\lambda_X$. As a consequence, the components of $C^\lambda$ map to components of $\mathcal A^\lambda_X$, and the nodes of $C^\lambda$ map to double intersections of these components. Since all points of nontrivial contact order are marked, therefore distinct, the contact orders are as prescribed.
\end{proof}

We detail the stable limit of a family of transverse curves in the expanded theory in the simplest nontrivial geometry. A discussion in the non-expanded theory may be found in~\cite[Section~4]{R15b}. 

\begin{example}\label{ex: line-degeneration}
We work with the target $X = \mathbb P^2$ and fix homogeneous coordinates $x,y,z$. The logarithmic structure is given by
\[
D = D_1\cup D_2 = \{x = 0\}\cup \{y = 0\}.
\]
Take the curve class to be that of a line in $\mathbb P^2$. Curves will carry two marked points $p_1$ and $p_2$, and the contact order of $p_1$ with $D_1$ is $1$ and with $D_2$ is $0$. Similarly, the contact order of $p_2$ with $D_2$ is $1$ and with $D_1$ is $0$. 

Work over the discretely valued field $\mathbb C(\!(t)\!)$ and examine the family of logarithmic stable maps given by 
\[
\mathbb V(x+y+tz)\subset \mathbb P^2. 
\]
The marked points are determined by the intersections with $D_1$ and $D_2$, and the maps are inclusions. The flat limit of this embedded family of curves over $\Spec \mathbb C[\![t]\!]$ in the dual projective space is the line $\mathbb V(x+y)$. Observe that this limit would violate the contact order condition since this curve passes through $D_1\cap D_2$. In the moduli space $\mathsf{ACGS}(\mathbb P^2)$, the limiting curve has two source components, each isomorphic to $\mathbb P^1$, joined at a node. The limiting map is
\[
C_0\cup C_1\to \PP^2,
\]
where the curve $C_0$ maps isomorphically onto $\mathbb V(x+y)$, and the component $C_1$ carries both markings and is collapsed to $[0,0,1]$. 

In the expanded theory, the target expansion is constructed by blowing up
\[
\PP^2\times\Spec\CC[\![t]\!]\to \Spec\CC[\![t]\!],
\]
at the point $([0,0,1],0)$. The special fiber is an expansion
\[
p: \mathbb F_1\cup \mathbb P^2\to \mathbb P^2,
\]
obtained as a union of $\mathbb P^2$ blown up at $[0,0,1]$, together with a copy of $\mathbb P^2$. The polyhedral decomposition corresponding to this target degeneration is given in Figure~\ref{fig: expansion-example} below. The exceptional divisor $E$ attaches to a line in $\mathbb P^2$. The map $p$ restricts to a proper surjective map on $\mathbb F_1$ and a constant map on $\mathbb P^2$. The limiting stable map is
\[
C_0\cup C_1\to \mathbb F_1\cup \PP^2\to \mathbb P^2,
\]
with the same source curve as before. The composite map to $\mathbb P^2$ is also the one described earlier. We describe the map to the expansion. The curve $C_0$ maps isomorphically onto the strict transform in $\mathbb F_1$ of the curve $\mathbb V(x+y)$ in $\mathbb P^2$. The curve $C_1$ maps to a line in $\mathbb P^2$, with $p_1$ and $p_2$ mapping to distinct toric boundary curves of this $\mathbb P^2$. The node of the curve maps to the double divisor in the reducible target.

The unexpanded logarithmic map, the minimal logarithmic structure on the base has characteristic $\mathbb N$. Equivalently, the tropical moduli space for curves of this type is $\RR_{\geq 0}$. Since this fan admits no nontrivial conical subdivisions, the expanded space and the logarithmic mapping space, which a priori differ by toroidal modification, are isomorphic in a neighborhood of this point. It is only the universal families that differ. 
\begin{figure}[h!]
\begin{tikzpicture}

\begin{scope}[shift = {(5,0)}]
\fill[white!70!violet, path fading=north] (0,0)--(0,2)--(2,2)--(2,0) -- cycle;
\fill[white!70!violet, path fading=south] (0,0)--(2,0)--(2,-1.414)--(-1.414,-1.414) -- cycle;
\fill[white!70!violet, path fading=west] (0,0)--(0,2)--(-1.414,2)--(-1.414,-1.414) -- cycle;
\end{scope}

\draw[->,violet] (5,0)--(5,2);
\draw[->,violet] (5,0)--(7,0);
\draw[->,violet] (5,0)--(3.686,-1.414);

\draw[->,violet] (6,1)--(6,2);
\draw[->,violet] (6,1)--(7,1);
\draw[-,violet] (6,1)--(5,0);


\draw [ball color=violet] (5,0) circle (0.5mm);
\draw [ball color=violet] (6,1) circle (0.5mm);
\end{tikzpicture}
\caption{The cone over the fan pictured above gives rise to a toric degeneration of $\mathbb P^2$ over $\A^1$. The fiber over $0$ is the degeneration described above. In the discussion above, the component $C_0$ maps to the target component dual to the central vertex, while $C_1$ maps to the component dual to the vertex on the upper right.}\label{fig: expansion-example}
\end{figure}

\end{example}


\subsection{The virtual class}\label{sec: vir-class} We have constructed a Deligne--Mumford stack $\Kbar^\lambda_\Gamma(X)$. There are three ways of constructing a virtual fundamental class.

\subsubsection{The first way} As a fibered category over logarithmic schemes,  $\Kbar^\lambda_\Gamma(X)$ can be identified with a subcategory of $\mathsf{ACGS}_\Gamma(X)$, consisting of those logarithmic stable maps to (an unexpanded) $X$ over $S$, such that the tropical moduli map
\[
S^{\trop}\to \mathsf{ACGS}(\Sigma),
\]
factors through the subcomplex $\Kbar^\lambda_\Gamma(\Sigma)$. Restrict the universal family to this subcategory to obtain a curve family $\mathcal C\to \Kbar^\lambda_\Gamma(X)$, and a morphism $\mathcal C\to X$. The relative obstruction theory of the map
\[
\Kbar_\Gamma^\lambda(X)\to \Kbar_\Gamma^\lambda(\mathcal A_X)
\]
is perfect. Indeed, the argument for this is identical to the one explained in~\cite[Section 6.1]{AW}. Given an $S$-point of $\Kbar_\Gamma^\lambda(X)$, let $S'$ be a strict infinitesimal extension, given by an ideal $I$. Consider the diagram of lifts below:

\[ \vcenter{\xymatrix{
& & X \ar[d]  \\
C \ar[r] \ar@/^15pt/[urr]^f \ar[d] & C' \ar[r] \ar@{-->}[ur] \ar[d] & \mathcal A_X \\
S \ar[r] & S' .
}} \]
Since the extension is strict, the tropical data of $C/S$ and $C'/S'$ are identical. The map $C\to \mathcal A_X$ extends automatically to $C'$. Lifts of these data are a torsor on $C$ of abelian groups $f^\star T^{\mathrm{log}}_X\otimes I$. We thus obtain an obstruction theory for $\Kbar^\lambda_\Gamma(X)$ over $\Kbar_\Gamma^\lambda(\mathcal A_X)$, using the formalism of~\cite[Section~5]{Wis11}. The virtual pull back of the fundamental class under the morphism $\Kbar_\Gamma^\lambda(X)\to \Kbar_\Gamma^\lambda(\mathcal A_X)$ yields a virtual fundamental class on $\Kbar_\Gamma^\lambda(X)$ by~\cite{Mano12}. 

\subsubsection{The midway} The first way ignores any subdivision that has been done to the universal family. That is, both curve and target are unbubbled. We may introduce bubbles in the curve while keeping the target, and the obstruction theory unchanged.

The universal modified curve admits a contraction 
\[
\mathrm{st}: \mathcal C^\lambda\to \mathcal C.
\] 
This morphism is a logarithmic modification, so it is a contraction of chains of rational bubbles. Consider the lifting problem for logarithmic maps $[g:C^\lambda\to X]$ from the modified curve $C^\lambda/S$, to a strict infinitesimal extension $S'$ given by an ideal $I$. Once again, the lifts are given by a torsor on $C^\lambda$ of the abelian groups $g^\star T^{\mathrm{log}}_X\otimes I$. By construction, the maps
\[
C^\lambda\to X
\]
factor through $C^\lambda\to C$. Since the bubbles are rational, pushing forward via $\mathrm{R}\mathrm{st}_\star(\cdot)$ identifies the torsors controlling the obstruction theory coming from $\mathcal C^\lambda$ with those on $\mathcal C$. Again, virtual pull back of the fundamental class from $\Kbar_\Gamma^\lambda(\mathcal A_X)$ gives rise to a virtual class, which by the discussion above is the same virtual class. 

\subsubsection{The second way} In the intermediate version above, we considered only the deformation theory of the collapsed map. Typically, one studies the obstructions to deformations of the map to the expansion. Again, the virtual class is unchanged. 

This follows by considering the lifting problems as above. Specifically, we examine the composition
\[
\mathcal C^\lambda\to \mathcal X^\lambda \to X\times \Kbar^\lambda_\Gamma(X)\to X,
\]
Since the second arrow is logarithmically \'etale over $\Kbar_\Gamma^\lambda(X)$ since it is the pullback of a subdivision. It therefore identifies relative logarithmic tangent bundles. In turn the relative logarithmic tangent bundle of the product is simply pulled back from the logarithmic tangent bundle of $X$. Once again, the torsors controlling lifts over strict extensions are naturally identified, and the virtual pullback produces a virtual class. In all cases, Manolache's results lead to the following~\cite{Mano12}.

\begin{theorem}
The Deligne--Mumford stack $\Kbar^\lambda_\Gamma(X)$ carries a virtual fundamental class in expected dimension. This class is equal to the virtual pullback of the ordinary fundamental class along the morphism
\[
\Kbar_\Gamma^\lambda(X)\to \Kbar_\Gamma^\lambda(\mathcal A_X).
\]
\end{theorem}

\subsection{Virtual birationality}\label{sec: vir-birationality} By construction, each of our moduli spaces come with a morphism
\[
\Kbar_\Gamma^\lambda(X) \to \mathsf{ACGS}_\Gamma(X).
\]
This may be seen in a moduli theoretic manner. Given a map to an expansion
\[
C^\lambda\to \mathcal X^\lambda\to X,
\]
over $S$, stabilize the underlying morphism to $X$ to obtain $[C\to X]$ over $S$, and equip it with the pushforward logarithmic structure~\cite[Appendix~B]{AMW12}. The logarithmic structure on $S$ may not be minimal, but by the universal property of minimality the map determines a minimal object. 

The different approaches to the relative theory give the same invariants, c.f.~\cite{AMW12}.

\begin{proposition}\label{prop: vir-birationality}
Pushforward under the morphism
\[
\Kbar_\Gamma^\lambda(X) \to \mathsf{ACGS}_\Gamma(X).
\]
identifies virtual classes. The logarithmic invariants of $X$ are computed by the expanded theory.
\end{proposition}

\begin{proof}
The result follows from the discussion above, using the ideas in~\cite[Section~6]{AW}. Indeed, there is a birational morphism
\[
\Kbar^\lambda_\Gamma(\mathcal A_X)\to \mathsf{ACGS}_\Gamma(\mathcal A_X).
\]
As discussed above we may construct virtual classes for maps to $\mathcal X^\lambda$ and minimal maps to $X$ relative to these two moduli spaces. However, the obstruction theory for $  \Kbar_\Gamma(X)\to \Kbar_\Gamma(\mathcal A_X)$ clearly pulls back to the obstruction theory for $  \Kbar^\lambda_\Gamma(X)\to \Kbar^\lambda_\Gamma(\mathcal A_X)$. The result follows from a theorem of Costello~\cite[Theorem~5.0.1]{Cos06}. The statement about Gromov--Witten invariants follows from the projection formula. 
\end{proof}

Cosmetic changes to the above argument also yield the following.

\begin{proposition}
Given any two moduli spaces $\Kbar_\Gamma^\lambda(X)$ and $\Kbar_\Gamma^\mu(X)$ of transverse stable maps to expansions and a logarithmic modification
\[
\Kbar_\Gamma^\lambda(X)\to \Kbar_\Gamma^\mu(X),
\]
pushforward identifies the virtual classes. 
\end{proposition}

\begin{remark}\label{rem: inverse-limits}
Since common refinements of subdivisions exist, the collection of logarithmic modifications of $\Kbar^\lambda_\Gamma(X)$ forms an inverse system. As a consequence of the proposition, there is a well-defined virtual class in the group $\varprojlim_\lambda A_\star(\Kbar^\lambda_\Gamma(X),\mathbb Q)$, where the inverse limit is taken under proper pushforward maps, see~\cite{Alu05} for a discussion of such inverse limit Chow groups. Note that the inverse limit is the logarithmic Chow \textit{homology} groups~\cite{Bar18}, rather the much smaller direct limit Chow cohomology groups examined in~\cite{MPS20}.
\end{remark}

\subsection{What is this a moduli space of?}\label{sec: really-a-moduli-space} Each of the spaces $\Kbar^\lambda_\Gamma(X)$ constructed here is a solution to a moduli problem over \textit{schemes}. Over logarithmic schemes, tautologically, they are defined as subcategories of $\mathsf{ACGS}_\Gamma(X)$ of families of curves over a base whose tropicalizations lie in the subdivision indexed by $\lambda$. Indeed, logarithmic modifications are always subcategories when considered in this manner, see~\cite[Definition~3.8]{Kato-LogMod} and~\cite[Section~2.5]{RSW17A}. This category admits minimal objects -- those families where at each closed point, the tropicalization map is an isomorphism onto a cone of the $\lambda$-subdivision. The category is therefore representable by a Deligne--Mumford stack~\cite[Appendix~B]{Wis16a}. It ignores the modifications done to the curve and target family. The same category also carries a second modular interpretation, coming from the universal modified curve and target. Once again, over logarithmic schemes, one studies families of logarithmic maps to an expansion of $X$ such that the universal curve and target are pulled back from the tropical families corresponding to the indexing element $\lambda$.
%
%
\newpage

{\Large \part{The gluing formula}}

A basic intuition guides this part of the paper. Let $\mathscr Z\to \A^1$ be a simple normal crossings degeneration, with general and special fibers $Z_\eta$ and $Z_0$. Given a curve $C_\eta$ in $Z_\eta$, we can specialize it to obtain a curve in $Z_0$. The reverse is a problem in deformation theory. Given a curve $C_0$ in $Z_0$, one needs to smooth it out over $\A^1$ in order to describe a curve in $Z_\eta$. In general this problem is much too hard, and one must either work in the logarithmic category or place hypotheses on the curve in the special fiber. A natural hypothesis is that the components of $C_0$ interact nicely with the strata of $Z_0$: each component meets only the codimension $1$ strata of $Z_0$ at points and is disjoint from the higher codimension strata. In this case, there is a \textbf{predeformability} condition: when two components meet at a point along a divisor of $Z_0$, the orders of vanishing must be equal. This condition is necessary for a smoothing to exist.

In simple cases, such as rational curves in toric varieties, predeformability is sufficient for deformability~\cite{MR16,NS06,R15b}.  The general problem is hopeless, but for the purposes of working with the virtual geometry of the space of maps, one pretends as if this condition suffices always, and Jun Li's relative Gromov--Witten theory is built from this point of view. When the special fiber of $\mathcal Z$ is not allowed to expand, the right component of the space of maps to the special fiber is the locus of logarithmic maps. However, the intuitive picture of gluing along components breaks down. One needs to study maps into the strata of the degeneration, which is where punctured logarithmic curves enter the picture~\cite{ACGS17}. Using the previous part of the paper, we reintroduce the transversality requirement, obtaining a proper space of maps to expanded targets, and prove the gluing formula following this picture. 

In order to work virtually, we work tropically. An inventive observation of Abramovich and Wise is that the space of maps to $\mathscr Z$ is virtually smooth over the space of tropical maps to the tropicalization~\cite{AW}. We take this observation to its extreme by first proving the gluing formula entirely within the combinatorial framework. We degenerate the cone complex $\Sigma_{\mathscr Z}$ into an extended cone complex -- a singular tropical object. Tropical nodal curves in the special fiber of the tropical degeneration can have nothing to do with maps to a general fiber, but we find that \textit{predeformability is tropical deformability}. After modifying the tropical targets to ensure combinatorial transversality, the nodal tropical curves may always be smoothed out leading to a tropical gluing formula for the cone complexes. Standard virtual pullback techniques imply the gluing formula. \\

\section{Extended tropicalizations}

Extended tropicalizations are compactifications of fans and polyhedral complexes obtained by adding faces at infinity which are themselves such complexes. We will use them to describe degenerations of tropical varieties to singular ones. A careful treatment of extended cone complexes can be found in~\cite[Section~2]{ACP}, and the relationship with Berkovich spaces and tropical geometry can be found in~\cite{Pay09,Thu07}. An interesting \textbf{pointification} viewpoint on the matter, considering piecewise linear functions taking the value infinity, has been introduced by Huszar--Marcus--Ulirsch~\cite{HMU}. This pointification is functorial, and sets the discussion below on sound categorical footing.

Let $\Sigma$ be a cone complex with integral structure. Each cone $\sigma\in\Sigma$ is obtained from its dual cone $S_\sigma$ of linear functions as the space of monoid homomorphisms
\[
\sigma = \Hom_{\bf Mon}(S_\sigma,\RR_{\geq 0}).
\]
By replacing $\RR_{\geq 0}$ with the extended monoid $\mathbf{R}_{\infty} = \RR_{\geq 0}\sqcup\{\infty\}$ equipped with its order topology, we obtain a compactification, the \textbf{extended cone}
\[
\sigma\hookrightarrow \overline \sigma = \Hom_{\bf Mon}(S_\sigma,\mathbf{R}_{\infty}).
\]
The cone complex $\Sigma$ can be described as a colimit of a diagram of cones, and by replacing each cone with its extended cone we obtain a canonical compactification
\[
\Sigma\hookrightarrow \overline \Sigma. 
\]
{
\begin{warning}\label{warning: many-bijections}
Pictures of the extended tropicalization can cause some consternation, as both order reversing and order preserving bijections are at play. The extended tropicalization of a toric variety $X$, with fan $\Sigma$, and dense torus $\mathbb G_m^r$ is equipped with an action of the vector group $\mathbb R^r$. The extended tropicalization is stratified by the orbits of this group action, and is in order preserving bijection with the orbits of the $\mathbb G_m^r$ action on $X$

There is an order reversing bijection between the \textit{closures} in the extended tropicalizations of the $k$-dimensional cones of $\Sigma$ and the codimension $k$ strata of $X$, see Figure~\ref{fig: ext-trop-P2}. Since each face at infinity can be interpreted as the extended tropicalization of a toric variety in the boundary of $X$, the order reversing and order preserving bijections interact in the boundary. While this may cause some initial confusion, it is an efficient tool in recording the necessary combinatorics for the degeneration problem. 
\end{warning}}

\begin{figure}
\includegraphics[scale=0.25]{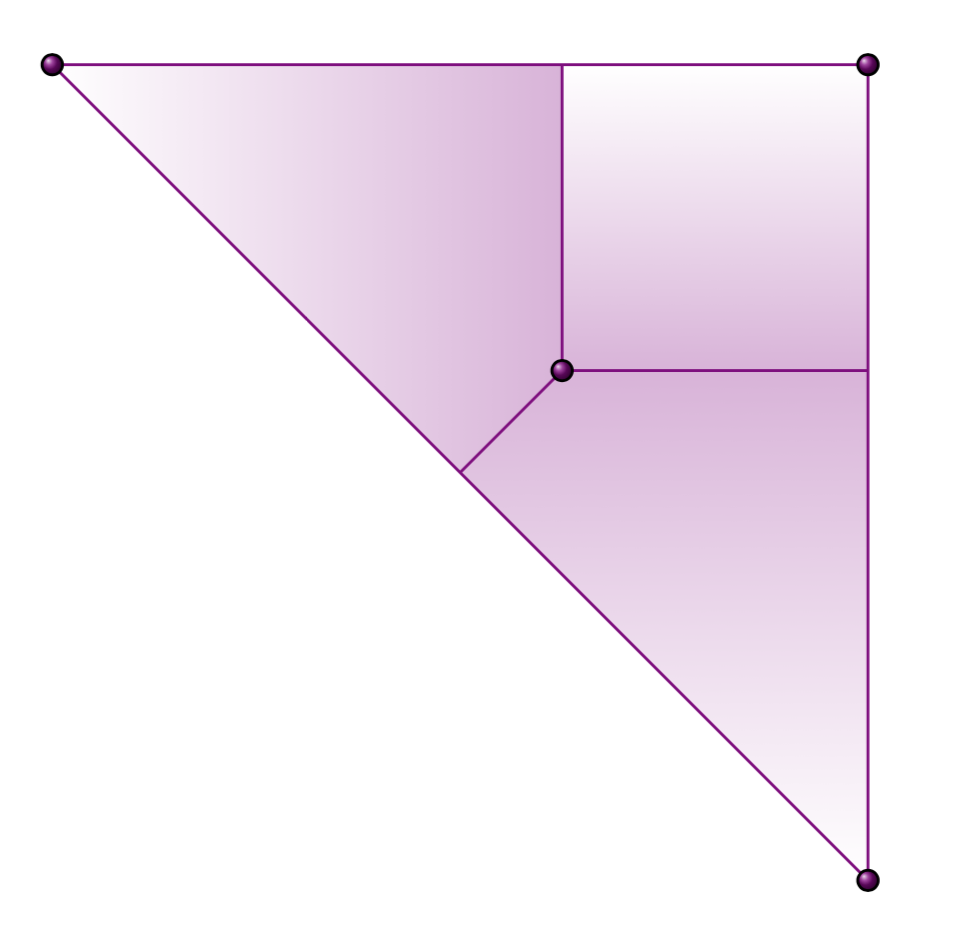}
\caption{The extended tropicalization of $\PP^2$ consists of $3$ extended cones. The face at infinity perpendicular to each ray is the extended tropicalization of the dual stratum. }
\label{fig: ext-trop-P2}
\end{figure}

Given a polyhedral complex $\mathscr P$ that is not necessarily a cone complex, one similarly obtains a compactification $\overline{\mathscr P}$ of the asymptotic directions as follows. Let $\Sigma_{\mathscr P}$ be the cone over the polyhedral complex. The height $1$ slice of the extended cone complex $\overline{\Sigma}_{\mathscr P}$ gives a compactification of $\mathscr P$

Extended cones can be glued along extended faces by taking colimits to obtain a larger class of objects, as has already been suggested in the literature~\cite[Remark~2.4.1]{ACP}. We will be concerned with the following three non-exclusive examples.
\begin{enumerate}[(A)]
\item {\bf Nodal tropical curves.} Given a pointed tropical curve $\plC$, its compactification is obtained by adding one point at infinity on each unbounded half-edge of $\plC$. By gluing such extended tropical curves along their infinite points one obtains objects that we call \textbf{nodal tropical curves}.
\item {\bf Boundaries.} Given a cone complex $\Sigma$, the complement $\overline \Sigma\setminus\Sigma$ is a union of the extended faces. This has an irreducible component decomposition into compactified cone complexes corresponding to the rays of $\Sigma$. 
\item {\bf Degenerate fibers.} Given a cone complex $\Sigma$ and a surjective map of cone complex $\Sigma\to \mathbb{R}_{\geq0}$, pass to the associated map $\overline\Sigma\to \Rb_\infty$. The fiber over $\infty$ is a union of extended cone complexes glued along boundary divisors. 
\end{enumerate}

A picture that encapsulates all three examples is given Figure~\ref{fig: tropical-degeneration}.

\begin{figure}[h]
\includegraphics[scale=0.45]{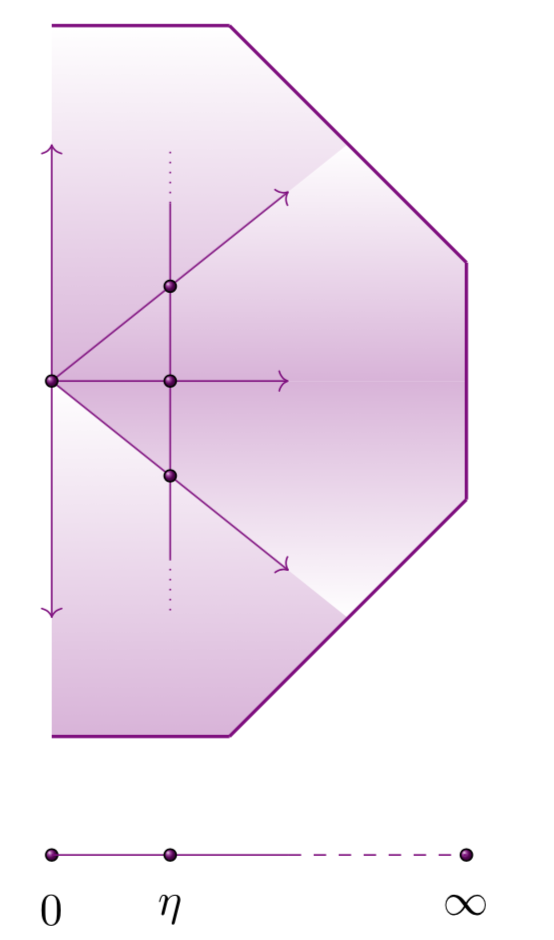}
\caption{The figure describes a tropical degeneration over $\Rb_\infty$ which arises as the tropicalization of a toric degeneration of $\PP^1$ over $\A^1$. The special fiber consists of three components, each equal to $\RR\sqcup \{\pm\infty\}$. Over a general point $\eta$ in $\RR$, the fiber is a subdivided real line with three vertices dual to the three components in the special fiber. }
\label{fig: tropical-degeneration}
\end{figure}

\section{The tropical gluing formula}\label{sec: tropical-gluing}

In this section, we prove a gluing formula entirely within combinatorial geometry. A degeneration of a cone complex $\overline \Sigma$ over $\mathbf{R}_{\infty}$ gives rise to a degeneration of the space of tropical maps. We describe the components of the degenerate moduli space by gluing maps from curves to the components of the degenerate fiber. 

\subsection{The tropical moduli space of maps} Let $\Sigma\to \RR_{\geq 0}$ be a surjective morphism of cone complexes and let $\overline \Sigma\to \Rb_{\infty}$ be the associated morphism of extended complexes. Let $\Sigma_r$ denote the fiber over the point $r\in \Rb_\infty$. We consider $\Sigma_\infty$ to be a degenerate limit of $\Sigma_0$. 

A family of tropical stable maps to a family $\Sigma\to \RR_{\geq 0}$, parameterized by a cone $\tau$, is a commutative square
\[
\begin{tikzcd}
\plC \arrow{r}{F}\arrow{d} & \Sigma\arrow{d} \\
\tau \arrow{r} & \RR_{\geq 0},
\end{tikzcd}
\]
where $\plC$ is a family of tropical curves, and $F$ is a family of tropical maps to $\Sigma$. We obtain a stack over the category of cones, whose fiber over $\tau$ is the groupoid of tropical stable maps over $\tau$. As before, we restrict our attention to the subcategory of maps of a fixed finite set of combinatorial types. The result is a finite type tropical moduli stack 
\[
\pi: \mathsf{ACGS}_\Gamma(\Sigma)\to \RR_{\geq 0}
\]
parameterizing such maps. The morphism $\pi$ presents the moduli space as a family over $\RR_{\geq 0}$. Let $\rho_i$ be a ray of the moduli space and let $v_i$ be its primitive integral generator. Then $\pi(v_i)$ generates a submonoid of some index in $\mathbb N$. Call this index $m_{\rho_i}$. 

Passing to extended tropicalizations gives rise to a morphism  
\[
\pi: \overline {\mathsf{ACGS}}_\Gamma(\overline{\Sigma})\to \Rb_{\infty}.
\]
The components lying above $\infty$ are in bijection with the rays $\rho_i$ that dominate $\RR_{\geq 0}$. As we soon describe, each such extended component is a moduli stack of tropical maps. As $\pi$ is a piecewise linear function, we expect a ``rational equivalence'' of the form
\begin{equation}\label{eq: rational-equivalence}
[\overline{\mathsf{ACGS}}_\Gamma(\overline \Sigma_0)] \ \aquarius \sum_{i} \frac{m_{\rho_i}}{\mathsf{Aut}(\rho_i)} [\overline {\mathsf{ACGS}}_{\rho_i}(\Sigma_\infty)].
\end{equation}
Here, $\overline {\mathsf{ACGS}}_{\rho_i}(\Sigma_\infty)$ is a finite cover of the component at infinity dual to $\rho_i$, obtained by removing generic automorphisms. We will not need to make the rational equivalence here precise, but if we use the categorical equivalence of cone complexes with Artin fans, a precise translation can be made in the Chow group of an appropriate family of Artin stacks~\cite{CCUW,U16}, and this implies the statement in~\cite[Section~4.1]{ACGS15}. 

Our next task is to describe the infinite faces $\overline {\mathsf{ACGS}}_{\rho_i}(\Sigma_\infty)$ in terms of maps to the strata. 

\subsection{Extended faces} Fix a ray $\rho$ of the tropical moduli space $\mathsf{ACGS}_\Gamma(\Sigma)$. We now describe a moduli space $\overline {\mathsf{ACGS}}_{\rho}(\Sigma_\infty)$ which is a finite cover of the infinite face of $\overline{\mathsf{ACGS}}_\Gamma(\Sigma)$ dual to $\rho$, as follows.

Choose a general point in the ray $\rho$ of ${\mathsf{ACGS}}_\Gamma(\Sigma)$. Since the multiplicity of vertical divisors will be equal to zero, we lose no generality in assuming that $\rho$ surjects onto $\RR_{\geq 0}$. Let the general point chosen map to $\eta\in \RR_{\geq 0}$. The fiber $\Sigma_\eta$ of the tropical degeneration is a polyhedral complex, and the fiber in $\mathsf{ACGS}_\Gamma(\Sigma)$ over $\eta$ intersected with $\rho$ is a possibly stacky point of the moduli space. This determines a single tropical map
\[
f_\eta: \plC_\eta\to \Sigma_\eta
\]
up to automorphisms. Since the point was chosen on a ray, the map is \textbf{rigid}: any deformation of the map $[f_\eta]$ changes the combinatorial type. We make the following additional assumption.

\begin{assumption}\label{assump: geneirc-transverse-trop}
{The curve $\plC_\eta$ is combinatorially transverse to $\Sigma_\eta$. That is, the image of every vertex of $\plC$ is a vertex of $\Sigma_\eta$ and the image of every edge is an edge.} 
\end{assumption}

\noindent
{\it Justification.} As described above, the ray $\rho$ determines a family of tropical maps, and we may subdivide $\Sigma$, inducing a subdivision of $\Sigma_\eta$ such that the transversality condition is met for the map $\plC_\eta\to \Sigma_\eta$, for all $\eta$. It follows immediately from the arguments in~\cite{AW} that this has the effect of replacing the moduli space with a modification. We are only interested in describing the component $\mathsf{ACGS}_\rho(\Sigma_\infty)$ up to modifications, so this will have no effect. 

We obtain a family of tropical maps to a family of targets, depending on a parameter $\eta$ in $\RR_{\geq 0}$. The curve $\plC_\eta$ has bounded edges whose lengths are linear functions in $\eta$. The special fiber of the tropical family is obtained by $\eta\to \infty$. The point at infinity of $\rho$ parameterizes a map
\[
f_\infty: \plC_\infty\to \Sigma_\infty.
\]

\begin{lemma}
The edges of $\plC_\infty$ are either (1) infinite rays associated to marked points, or (2) infinitely long edges, identified with the singular metric space $\Rb_{\infty}\sqcup_\infty \Rb_{\infty}$.
\end{lemma} 

\begin{proof}
If $\plC_\infty$ contained bounded edges of finite length, then by uniformly scaling all bounded edges, we produce a one-parameter deformation of the tropical map, which cannot exist by the rigidity hypothesis. 
\end{proof}

\noindent
{\bf Moduli at the extended face.} Fix a single tropical map $f_\infty: \plC_\infty\to \Sigma_\infty$. Consider the moduli space $\mathsf{ACGS}_\rho(\Sigma_\infty)$ parameterizing tropical stable maps 
\[
g_\infty: \plD\to\Sigma_\infty
\]
from nodal tropical curves to $\Sigma_\infty$, that (1) are deformations of a fixed map $f_\infty: \plC_\infty\to \Sigma_\infty$, and (2) come equipped with a uniform rescaling to $[f_\infty]$. Such maps are parameterized by a stack over the category of cones. Forgetting the rescaling information, we obtain a cover of the infinite face of $\mathsf{ACGS}_\Gamma(\Sigma)$ that is dual to $\rho$, as in Equation~(\ref{eq: rational-equivalence}). 


\subsection{The cutting morphism} The special fiber $\Sigma_\infty$ is glued from infinite faces. Each ray of the target $\Sigma$ determines a single extended face that is the canonical compactification of a cone complex. The special fiber of $\overline \Sigma$ is glued from these extended faces:
\[
\Sigma_\infty = \bigcup_v \overline \Sigma_{v} 
\]
where $v$ runs over the vertices. Given a map
\[
\plC\to \Sigma_\infty,
\]
cut the target $\Sigma_\infty$ into the pieces as above. For any map parameterized by a point of the moduli space $\mathsf{ACGS}_\Gamma(\Sigma_\infty)$, we correspondingly split the curve into components, at its infinite bounded edges. After deleting the extended faces, we obtain maps $\{\plC_v\to \Sigma_{v}\}_v$. Note that these tropical curves can in principle be disconnected. This produces a cutting morphism
\[
\mathsf{ACGS}_\rho(\Sigma_\infty)\to \prod_v \mathsf{ACGS}_{\rho}(\Sigma_{v}). 
\]
where $\mathsf{ACGS}_{\rho}(\Sigma_{v})$ is the moduli stack of tropical maps to the cone complex $\Sigma_v$ with the discrete data induced by $\rho$, by the star at $v$. The star at $v$ can be identified in an elementary fashion the fan of outwards directions at $v$ in the the polyhedral complex obtained by taking the fiber of $\Sigma\to \RR_{\geq 0}$ over the point $1$. It corresponds to an irreducible component in the special fiber of an associated geometric degeneration. 

\subsection{Modifying the product}\label{sec: modify-product} The stacks of maps $\mathsf{ACGS}_{\rho}(\Sigma_{v})$ can be used to describe the stack $\mathsf{ACGS}_\rho(\Sigma_\infty)$, at least up to modifications. The product has the following universal diagram.
\begin{equation}\label{eq: product}
\begin{tikzcd}
\bigcup_v\plC_v\arrow{rr}\arrow{dr} & & \Kbar\times \bigcup_v \Sigma_{v}\arrow{dl} \arrow {r} &  \bigcup_v \Sigma_{v}\\
&\Kbar: = \prod_v \mathsf{ACGS}_{\rho}(\Sigma_{v}) & &
\end{tikzcd}
\end{equation}
Let $G_\rho$ be the graph given by the source of the map $[f_\infty]$. Each edge $e$ of $G_\rho$ has vertices $v_1$ and $v_2$ and determines a half-edge incident to each. In the component $\Sigma_{v_1}$ containing $v_1$, the point at infinity of this half edge $e_1$ maps to a closed boundary stratum which we denote $\Sigma_e$, noting that the other half edge maps to the same stratum. This fits into an evaluation morphism
\[
\mathsf{ev}_{e}: \prod_v \mathsf{ACGS}_{\rho}(\Sigma_{v})\to \Sigma_e^2. 
\]
Note that in order to define this evaluation, we pass to the relative extended tropicalization in the universal curve, adding points at infinity, and inspect their images in $\Sigma_e$. This reflects the fact that these boundary evaluation morphisms are not generally logarithmic when restricting the logarithmic structure on the target to a stratum, but become logarithmic when the generic logarithmic structure along the divisor is removed.

\begin{remark} The naive strategy would be to now range over all the edges, pulling back the appropriate diagonal loci to obtain maps from collections of curves that glue to form a map to the target. However, the diagonal loci do not impose transverse conditions, and when this procedure is applied to moduli points in extended faces of the product, there is no natural way in which to smooth out the singular tropical curve. This reflects the geometric situation that the evaluations may map to small strata. We subdivide the target over the moduli space until the evaluations always map to vertices, or geometrically to top dimensional strata. This ensures that the gluing always imposes the expected number of geometric conditions, and smoothability follows.
\end{remark}

The required moduli space modifying the product is a twist on Theorem~\ref{thm: main-comb-thm}, but we first discuss the necessity of the modification in greater detail. At first approximation, we must expand each tropical target for the maps $\plC_v\to \Sigma_v$, which can be achieved by subdividing each factor as in the first part. Such a modification does not in immediately suffice, however. Such expansions of two targets are independent of each other, and there is no guarantee that the resulting expansions can be glued together. This is exemplified in Figure~\ref{fig: product-modification}. 

\begin{example} Let $X$ be the mirror dual to $\PP^1\times \PP^1$, whose fan is the fan over the faces of the square $(\pm 1,\pm 1)$. Consider the family of polyhedral complexes in $\RR^2$ given at time $t\in \RR_{\geq 0}$ by a bounded cell -- the square $(\pm t,\pm t)$ -- and rays beginning at the four points of this cell, with directions $(\pm 1,\pm 1)$ as depicted in the figure. This polyhedral complex forms a family $\Sigma\to \RR_{\geq 0}$. The toric degeneration described by this family degenerates $X$ to a union of $4$ copies of $\PP^2$.

The associated tropical degeneration consists of $4$ copies of $\PP^2_{\trop}$ glued along faces, as shown in Figure~\ref{fig: product-modification}. In understanding this picture, Warning~
\ref{warning: many-bijections} is best kept in mind. The picture depicts $4$ partially compactified vector spaces, whose origins are the four points of the square drawn. The cross consists of $4$ copies of $\RR\sqcup\{\infty\}$, glued at the infinite points. The points where the dotted arrows meet the cross are the origins of compactified vector spaces $\RR\sqcup\{\infty\}$. 

In the figure, the dashed arrows in the northwest and southwest quadrants are a non-transverse tropical curves, corresponding to a point in the product. These two target components must be modified in order to maintain transversality. However, when an infinite edge is added to the interior, it forces a subdivision at infinity. In this case, the blue dots on the cross itself show how subdivisions in the interior force subdivisions at infinity. 
\begin{figure}[h!]
\includegraphics[scale=0.23]{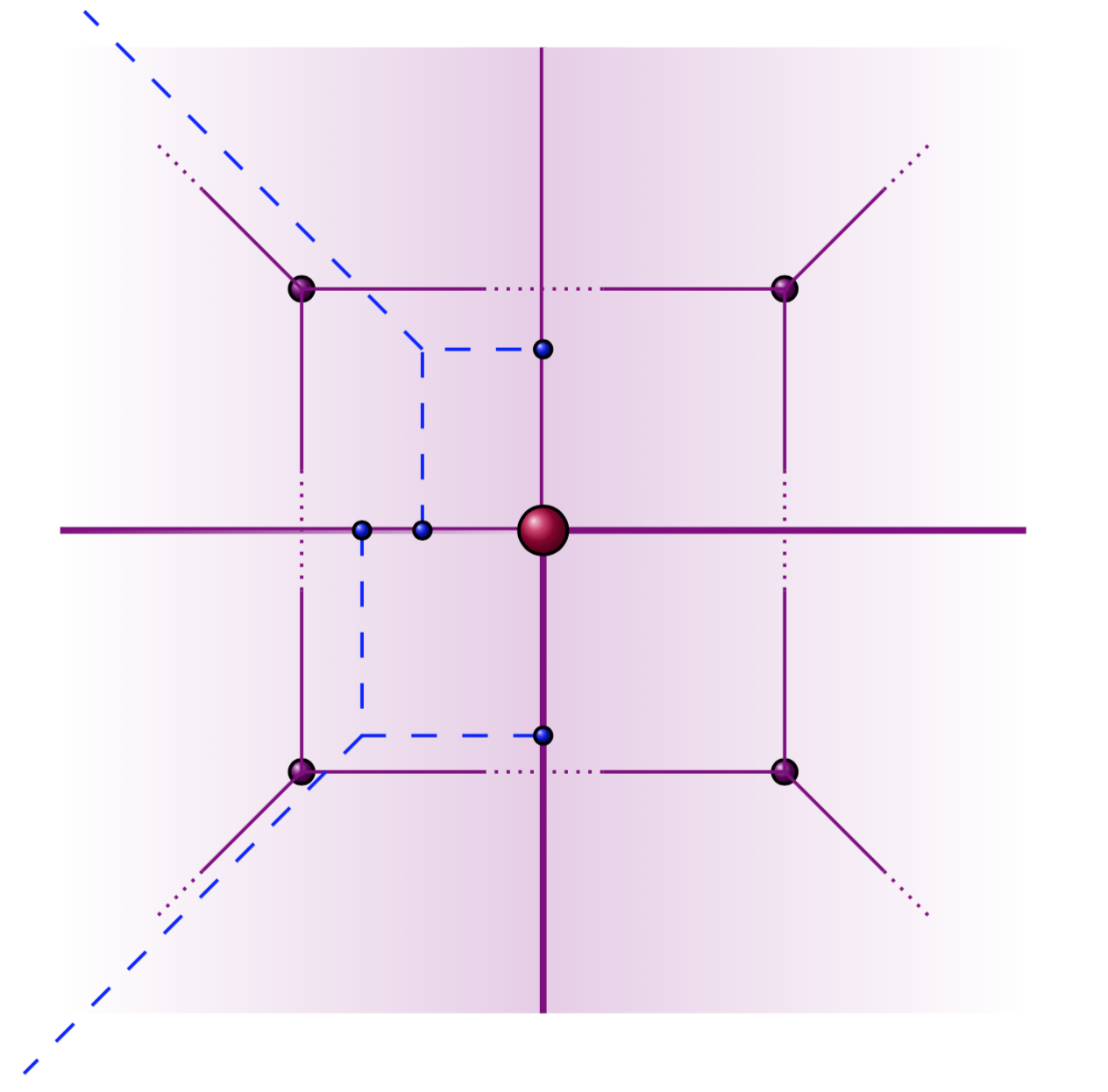}
\caption{Depicted here is the extended tropicalization of four copies of $\PP^2$ glued along their common boundaries. The axes surrounding the center are boundary curves along which the surfaces meet, while the four smaller vertices in the interior are dual to the components themselves. Two dashed tropical curves are drawn, indicating how transversality in one component forces subdivisions of another.}\label{fig: product-modification}
\end{figure}

The figure above indicates the general strategy. Subdividing the interior to accommodate an infinite edge of the curve forces a subdivision at infinity in that component. We choose a matching subdivision on the other side of the cross. Since marked points in our spaces always evaluate to the interiors of pairwise intersections of irreducible components, this can always be done. To make the appropriate construction, recall that we have inclusions of the infinite faces
\[
 \iota: \bigcup_v \overline \Sigma_{v}\to \Sigma_\infty \hookrightarrow \overline \Sigma.  
\]
Composing the universal map over $\prod_v \mathsf{ACGS}_{\rho}(\Sigma_{v})$ with the morphism $\iota$ we think of the universal target\footnote{In fact, we only use the partial compactification of $\Sigma$ determined by deleting the extended faces corresponding to cones of $\Sigma$ of dimension $2$ and higher.} as being a constant family of $\overline \Sigma$.
\end{example}

The required variation of Theorem~\ref{thm: main-comb-thm} in the first part is as follows.


\begin{theorem}\label{thm: var-on-main-comb-thm}
There exists a modification 
\[
\Kbar^\lambda={\bigtimes_v}^\lambda \Kbar_{\rho}(\Sigma_{v}) \to \prod_v \mathsf{ACGS}_{\rho}(\Sigma_{v})
\]
of the product, and modifications of its source and target family, and a moduli diagram
\[
\begin{tikzcd}
\bigcup_v\plC^{\lambda}_v\arrow{rr}\arrow{dr} & & \overline\Sigma_{\mathrm{exp}}^\lambda \arrow{dl} \arrow {r} &  \overline \Sigma\\
&\Kbar^\lambda & &
\end{tikzcd}
\]
with the following properties.
\begin{enumerate}[(1)]
\item {\bf Gluability.} The extended cone complex $\overline \Sigma^\lambda_{\mathrm{exp}}$ is the canonical compactification of a cone complex $\Sigma^\lambda_{\mathrm{exp}}$ that is combinatorially equidimensional over $\mathsf K^\lambda$. 
\item {\bf Curve family.} The universal curves $\plC^\lambda_v$ are each a family of tropical curves over the base $\Kbar^\lambda$. 
\item {\bf Degenerate transversality.} Over a point $t$ of $\Kbar^\lambda$ and for each $v$, the universal curve $\plC_v[t]$ maps to the relative interior of an extended face $\Sigma_{\mathrm{exp},v}[t]$, and is combinatorially transverse: the image of every polyhedron of the curve is a polyhedron of the target.
\end{enumerate}
\end{theorem}

\begin{proof}
We focus first on the final requirement, demanding that each curve maps to the relative interior of a maximal infinite face of $\overline \Sigma$, but is combinatorially transverse within that stratum. Fix a vertex $v$ and consider the map $\plC_v$ to the relative interior of an infinite face $F_v$ of $\Sigma_\infty$. Introduce the notation 
\[
\mathsf{PROD} = \prod_v \mathsf{ACGS}_{\rho}(\Sigma_{v})
\] 
to denote the product of moduli spaces attached to the vertices. To ensure transversality of $\plC_v$, we take $F_v\times\mathsf{PROD}$ and make subdivisions of the total space, later modifying the curve and base to ensure the requisite conditions. For any fan, the collection of star subdivisions is cofinal in the inverse system of all subdivisions, so we may assume that transversality for $\plC_v$ is achieved by making a sequence of star subdivisions of $F_v\times\mathsf{PROD}$. However, as outlined in the aforementioned example, the subdivisions of $F_v$ may force subdivisions of other extended faces $F_u$. To see that a consistent choice can be made to ensured while maintaining gluability of the target, we instead make the subdivisions of the total space of the degeneration $\Sigma$, that induce the given subdivision of $F_v$ upon passing to canonical compactifications.  

Given a cone $\tau$ of the fan $F_v\times\mathsf{PROD}$ that we wish to perform a star subdivision of, observe that it determines a stratum of the canonical compactification $\overline F_v\times\overline{\mathsf{PROD}}$. This is also a stratum of the larger compactification $\overline \Sigma\times\overline {\mathsf{PROD}}$ and thus determines a cone $\tau^{\mathrm{lift}}$ on the interior $\Sigma\times\mathsf{PROD}$. The cone $\tau$ is a cone contained in a face at infinity of $\tau^{\mathrm{lift}}$. Further subdivisions do not affect the question at this stage, so assume that $\tau^{\mathrm{lift}}$ is a simplicial cone inside a simplicial fan $\Sigma\times\mathsf{PROD}$. Perform the star subdivision of $\tau^{\mathrm{lift}}$.  Now passing to its canonical compactification, the induced subdivision on $\overline F_v\times\overline{\mathsf{PROD}}$ is the star subdivision of  $\overline F_v\times\overline{\mathsf{PROD}}$. This is the extended face corresponding to the strict transform. By performing a sequence of these star subdivisions, we obtain a family of expansions of the faces at infinity, expanded compatibly, as well as a family of expansions of the tropical target $\Sigma$. 

We now proceed as in Theorem~\ref{thm: main-comb-thm}. Pull back the resulting expansions of the faces $F_v\times\mathsf{PROD}$ to the curves $\plC_v$, flatten the source and target families, and obtain the modification
\[
{\bigtimes_v}^\lambda \Kbar_{\rho}(\Sigma_{v}) \to \prod_v \mathsf{ACGS}_{\rho}(\Sigma_{v})
\]
with the requisite transversality properties.
\end{proof}

\begin{remark}
When the boundary is a single smooth divisor, the divisor itself does not admit nontrivial modifications, as its tropicalization is a single point. As a result, the birational modification of the product is not necessary, and this accounts for the lack of modification of the product in traditionally considered cases~\cite{CheDegForm,KLR,Li02}.
\end{remark}

As with the construction of the spaces in the previous part of the paper, $\lambda$ records a choice, and there is an inverse system of such moduli spaces obtained by common refinement. 

\subsection{Evaluations and tropical gluing} The next two steps are not strictly necessary for the geometric arguments, but are combinatorial analogues of key steps in the next section. We first glue maps to form maps out of nodal tropical curves, and then smooth them to maps from ordinary tropical curves.

Fix the rigid tropical type, and consider maps that are deformations of
\[
f_\infty:\plC_\infty\to \Sigma_\infty,
\]
equipped with a fixed contraction to this map. Let $G$ be the graph formed by considering the vertices and infinite bounded edges of $\plC_\infty$. Previously, we constructed a cutting morphism
\[
\mathsf{ACGS}_\rho(\Sigma_\infty)\to \prod_v \mathsf{ACGS}_\rho(\Sigma_v).
\]
We also built a modification of the product with a transverse disconnected universal curve. The cutting morphism from $\mathsf{ACGS}_\rho(\Sigma_\infty)$ does not immediately lift to the modified product. However, we may subdivide $\mathsf{ACGS}_\rho(\Sigma_\infty)$ and its universal curve and target to obtain a space of transverse maps $\Kbar^\lambda_\rho(\Sigma_\infty)$ and a new cutting morphism
\[
\kappa: \Kbar^\lambda_\rho(\Sigma_\infty)\to {\bigtimes_v}^\lambda \Kbar_\rho(\Sigma_v).
\]
We wish now to identify the image of this map. For a morphism to the product to come from a morphism to $\Sigma_\infty$, two conditions are apparent:
\begin{enumerate}[(A)]
\item {\bf Continuity.} For each edge $e$ of $G$ between $v_1$ and $v_2$, the images of the two markings incident to $v_1$ and $v_2$ comprising $e$ must coincide. 
\item {\bf Smoothability.} The construction of $\Kbar^\lambda_\rho(\Sigma_\infty)$ as a component of the special fiber of a tropical degeneration of maps, demands that any map must arise as the limit of a $\RR_{> \eta}$-family of maps to $\Sigma$. 
\end{enumerate}

Over the modified product moduli space we have an expanded target family $\Sigma^\lambda_{\mathrm{exp}}$. An edge $e$ in $G$ determines a choice of codimension $2$ stratum in the extended faces of this expanded target family. The infinite point of a marked end of the curve $\plC_v$ containing one end of $e$ has an evaluation morphism to the relative interior of the corresponding stratum of the target. We denote this relative divisor by $\mathrm{D}_e(\Sigma^\lambda_{\mathrm{exp}})$. Since there are two vertices at the ends of each edge $e$, we have an evaluation morphism
\[
\mathsf{ev_e}: {\bigtimes_v}^\lambda \Kbar_\rho(\Sigma_v)\to \mathrm{D}^{\{2\}}_e(\Sigma^\lambda_{\mathrm{exp}}).
\]
The symbol on the right is the fibered second power of $\mathrm{D}_e(\Sigma^\lambda_{\mathrm{exp}})$ over the moduli space ${\bigtimes_v^\lambda} \Kbar_\rho(\Sigma_v)$. The subdivision has been chosen such that the image of a boundary marking is always a vertex of  $D_e(\Sigma^\lambda_{\mathrm{exp}})$. We have a relative diagonal morphism
\[
\Delta_e: \mathrm{D}_e(\Sigma^\lambda_{\mathrm{exp}})\hookrightarrow \mathrm{D}^{\{2\}}_e(\Sigma^\lambda_{\mathrm{exp}}).
\]
Note that in each fiber over the moduli space, the image of either evaluation morphism attached $e$ is a {vertex} of the divisor. The intersection of the image of $\Delta_e$ and the image of the map
\[
\mathsf{ev_e}: {\bigtimes_v}^\lambda \Kbar_\rho(\Sigma_v)\to \mathrm{D}^{\{2\}}_e(\Sigma^\lambda_{\mathrm{exp}}).
\]
is a subcomplex of $\mathrm{D}_e(\Sigma^\lambda_{\mathrm{exp}})$. By ranging over all edges, we obtain a tropical moduli space\footnote{In full analogy with the geometric situation, this intersection of diagonals a stand in for the refined intersection product for regular embeddings, in the sense defined by Fulton.}
\[
\mathcal G^\lambda(\Sigma_\infty) = \bigcap_e \mathrm{im}(\Delta_e)\cap {\bigtimes_v}^\lambda \Kbar_\rho(\Sigma_v).
\]
These are tropical maps to expansions of $\Sigma_\infty$ that are transverse to the target and whose evaluation morphisms agree. 

\subsection{Tropical smoothing}\label{sec: trop-smoothing} A final step remains to complete a tropical degeneration formula -- the singular tropical curves constructed by gluing must be smoothed out. This step reflects the geometric phenomenon of having to build logarithmic enhancements of glued maps, which we discuss in the next section. In both steps, transversality of the curve is paramount. 

Given a $p$ point of $\mathcal G^\lambda(\Sigma_\infty)$, we obtain a map
\[
g_\infty: \plD_\infty\to \Sigma^\lambda_\infty.
\]
We wish to smooth it out to a family $\plD_\eta\to \Sigma_\eta$ for $\eta\in \RR_{\geq 0}$, for $\eta$ sufficiently large, whose limit is the map $[g_\infty]$. This involves ``shrinking'' the infinite edge lengths of $\plD_\eta$ to produce a continuous tropical map to $\Sigma_\eta$. 

Consider a nodal edge $e$ of $\plD_\infty$. Its image under $g_\infty$ is an infinitely long edge of $\Sigma_\infty[p]$. However, over the point $p$, we have a canonical inclusion
\[
\Sigma_\infty[p]\hookrightarrow \overline \Sigma^\lambda_{\mathrm{exp}}[p]\subset \overline\Sigma_{\mathrm{exp}}^\lambda
\]
Smoothing out $g_\infty$ now amounts to resetting the infinite edge of $\plD_\infty$ lengths to finite values varying linearly in $\eta$, while maintaining continuity.

Given a nodal edge $e$ of the curve, the image $g_\infty(e)$ is a nodal edge of the target. Over $\eta$, this nodal edge of the target determines a unique bounded edge $\Sigma_\eta[p]$. This edge has length $\ell_e$. Let $m_e$ denote the expansion factor on the edge $e$ of $\plD_\infty$. Reset the edge length of $\plD_\eta$ to be $\ell_e/m_e$. Consider the resulting map
\[
\plD_\eta\to \Sigma_1[p].
\]
We claim that it is continuous and piecewise linear. To see this, observe that each edge of $\plD_\eta$ maps onto an edge of $\Sigma_1[p]$ by a linear map with slope equal to the expansion factor $m_e$. When two edges that meet at a vertex of $\plD_\eta$, that vertex is mapped to the corresponding vertex in $\Sigma_1[p]$ by construction. This implies continuity and piecewise linearity. In particular, we conclude that for any appropriate choice of $\lambda$, the morphism
\[
\Kbar^\lambda_\rho(\Sigma_\infty)\to \mathcal G^\lambda(\Sigma_\infty)
\]
is a piecewise linear homeomorphism.


\section{The virtual gluing formula}

In this final section, we prove the degeneration formula. We pull back the subdivisions constructed in the previous section to the moduli spaces of maps to the Artin fan, and then to the degeneration. Two additional ingredients yield the gluing formula, both of which are reminiscent of the traditional case~\cite[Section~7]{CheDegForm}. The first is the calculation of the number of logarithmic lifts of a glued, transverse map. The second is a comparison of obstruction theories. 

\subsection{The degeneration setup} Let $\mathscr Y\to \A^1$ be a toroidal degeneration. There exists a stack of logarithmic stable maps to the fibers of the degeneration $\mathsf{ACGS}_\Gamma(\mathscr Y)$. Let $Y_\eta$ be the general fiber and $Y_0$ the special fiber, equipped with its logarithmic structure. The moduli space comes equipped with a structure map 
\[
\mathsf{ACGS}_\Gamma(\mathscr Y)\to \A^1.
\]
The rays of the tropical moduli stack $\mathsf{ACGS}_\Gamma(\Sigma_{\mathscr Y})$ discussed in the previous section are indexed by finitely many \textbf{rigid tropical maps}. These determine divisors in the space of maps to the Artin fan, and hence virtual divisors of the stack above. We freely use the twofold interpretation of a ray $\rho$: as a ray in the tropical moduli space, and as a combinatorial type of a rigid tropical map. Note that this combinatorial type includes a finite marked graph, genus splitting, degree splitting, and contact order data~\cite[Definition~2.5.3]{ACGS15}. 

In the previous section we made the assumption that each rigid tropical map is supported on a subcomplex of $\Sigma_{\mathscr Y}$. The parallel assumption is made here. 

\begin{assumption}{For each rigid tropical map $\rho$, we assume that the virtual codimension $0$ strata of the component with type $\rho$ correspond to maps
\[
C\to Y_0
\]
that are transverse to the strata of $Y_0$.} 
\end{assumption}

This can always be done after replacing $\mathscr Y$ with a toroidal modification such that the rigid tropical maps are transverse to the tropicalization of the degeneration. Indeed, the image of each rigid tropical map gives rise to a polyhedral complex in the dual complex obtained as the fiber over $1$ in the tropical family of targets, denoted $\Sigma_{\mathscr Y}[1]$. After a polyhedral subdivision of $\Sigma_{\mathscr Y}[1]$, this rigid tropical map can be made transverse -- i.e. such that the image of every vertex (resp. edge) is a vertex (resp. edge) of this subdivided complex. This induces a logarithmic modification of $Y_0$, which achieves the requisite transversality claim. See the parallel Assumption~\ref{assump: geneirc-transverse-trop} in the previous section. 

\begin{remark}
Note that when $\mathscr Y$ is replaced with a logarithmic modification $\mathscr Y'$, the virtual modification of the moduli space $\mathsf{ACGS}_\Gamma(\mathscr Y')$ will acquire new rigid tropical curves. However, these new exceptional virtual divisors will not contribute to the Gromov--Witten theory numerically. Since any integral of interest will be pulled back from from $\mathsf{ACGS}(\mathscr Y)$, the projection formula will guarantee that these new virtual divisors will be zeroed out. Of course, if insertions are not pulled back, this ceases to be true. Nevertheless, once the numerical data is fixed, there are only finitely many rigid curves to sum over. We may always perform a logarithmic modification to $\mathscr Y$ to ensure that transversality holds for maps with each of these fixed rigid types. 
This is essentially the strategy employed in~\cite{NS06}, though in that case the rigid tropical curves are  zero dimensional. 
\end{remark}

The \textit{decomposition theorem} of Abramovich--Chen--Gross--Siebert equates the virtual class of maps to the special fiber with a sum of pieces,
\[
[\mathsf{ACGS}(Y_\eta)] \ = \sum_{i} \frac{m_{\rho_i}}{\mathsf{Aut}(\rho_i)} [\mathsf{ACGS}_{\rho_i}(Y_0)] \in A_\star(\mathsf{ACGS}(\mathscr Y)),
\]
in parallel with Equation~\ref{eq: rational-equivalence}. The space $\mathsf{ACGS}_{\rho_i}(Y_0)$ is a moduli space of logarithmic stable maps to $Y_0$, equipped with a uniform rescaling to the fixed rigid type $\rho_i$. See~\cite[Theorem~1.1.2]{ACGS15}. 

\subsection{The generic case: transverse gluing} Fix a rigid tropical curve $\rho$ with graph $G$. Write $Y_0$ as a union of irreducible components $X_{v}$ for $v$ a vertex of $G$. Our goal is now to describe the moduli space $\mathsf{ACGS}_{\rho}(Y_0)$, up to toroidal modifications, in terms of simpler pieces. Over the loci of generic curves in the Artin fan (i.e. over the vertex in the tropical moduli space), this is straightforward.  The (possibly empty) open stratum 
\[
\mathsf{ACGS}^\circ_{\rho}(Y_0)\subset \mathsf{ACGS}_{\rho}(Y_0)
\] 
of maps 
\[
C\to Y_0
\]
whose tropicalization is exactly $\rho$ is exactly this locus of generic maps. Recall we are assuming by modifications that such generic curves are supported away from higher codimension strata in $Y_0$. Consider the target normalization
\[
\bigcup X_{v}\to Y_0. 
\]
We therefore obtain a cutting morphism
\[
\kappa^\circ: \mathsf{ACGS}^\circ_{\rho}(Y_0)\to \prod_v \mathsf{ACGS}_{\rho}^\circ(X_{v}). 
\]
For every edge $e$ of $G$, there are a pair of scheme theoretic evaluation morphisms at the corresponding marked points, giving us
\[
\mathsf{ev}: \prod_v \mathsf{ACGS}_{\rho}^\circ(X_{v})\to \prod_e (D_e^2)^\circ,
\]
where the genericity assumption tells us that the scheme theoretic evaluations all map to the dense open stratum in each divisor. The diagonal copy of $\prod_e D_e^\circ$ in the square is a regular embedding since the interiors of these divisors are smooth. The scheme theoretic fiber product gives rise to a new space
\[
\begin{tikzcd}
\mathcal G^\circ \arrow{r}\arrow{d} & \prod_v \mathsf{ACGS}_{\rho}^\circ(X_{v})\arrow{d}{\mathsf{ev}} \\
\prod_e D_e^\circ\arrow[swap]{r}{\Delta^\circ} & \prod_e (D_e^2)^\circ.
\end{tikzcd}
\]
The constructions above yield
\[
\kappa^\circ\times \mathsf{ev}: \mathsf{ACGS}^\circ_\rho(Y_0)\to \mathcal G^\circ.
\]
This morphism is finite and \'etale. A moduli point of $\mathcal G^\circ$ is a logarithmic map from a disjoint union of curves to the components of $Y_0$ that scheme theoretically glue. As the curve is transverse to the strata, the scheme theoretic data has a logarithmic enhancement, and in fact, finitely many such enhancements~\cite[Theorem~1.1]{Wis16b}. Since $\Delta^\circ$ is a regular embedding, the stack $\mathcal G^\circ$ also inherits a virtual fundamental class from the product. The virtual class on the space $\mathsf{ACGS}^\circ_\rho(Y_0)$ can be pushed forward to $\mathcal G^\circ$, and the two classes agree on $\mathcal G^\circ$ up to a finite multiple.

\begin{remark}
The moduli spaces above are generally non-proper, and the maps parameterized by deeper strata in their natural compactifications $\mathsf{ACGS}_\rho(Y_0)$ are non-transverse. In certain special situations, invariants can be defined on the non-proper moduli spaces, or be supported on the locus of transverse maps. In such situations, a degeneration formula can be proved. This is the situation in each of the papers~\cite{Bou19, CJMR, GPS, MR16, NS06}. The next section uses expansions to achieve such transversality in general. 
\end{remark}

\subsection{The general case: transversality via expansions}\label{sec: general-gluing} The antecedent paragraphs perform gluing in the locus that is dual to the tropically trivial part of the moduli space. We now extend the picture to the full moduli space. The modified spaces have been engineered in order to ensure that the transversality properties used in the case above are present at every point, so the arguments extend in a natural way.

We cease working with the Abramovich--Chen--Gross--Siebert space and replace $\mathsf{ACGS}_\rho(Y_0)$ with a moduli stack $\Kbar^\lambda_\rho(Y_0)$ over $\spec(\NN\to\CC)$  defined as the stack parameterizing logarithmic stable maps to families of expansions of $Y_0$ with discrete data $\rho$, where $\lambda$ is any of the elements in the inverse system of refinements constructed in Section~\ref{sec: maps-from-subs}. Consider a stable map over $S$
\[
\begin{tikzcd}
C\arrow{dr}\arrow{rr} & & \mathcal Y^\lambda \arrow{r} \arrow{dl} & Y_0\\
&S&
\end{tikzcd}
\]
that is transverse to the strata. Normalize $Y_0$ and observe that each component $X_v$ comes equipped with a collection of divisors. These are either of irreducible components of the normalization with divisors of $Y_0$ or preimages of the double divisors of $Y_0$. Equip each irreducible component with the corresponding divisorial logarithmic structure. Similarly, we partially normalize $\mathcal Y^\lambda$ into pieces $\mathcal X^\lambda_{v}$, noting that each $\mathcal X_v^\lambda$ is a degeneration of the corresponding component of $Y_0$. The curve also partially normalizes to produce, for each $v$, a morphism over $S$
\[
f_i: C_{v}\to \mathcal X^\lambda_{v}\to X_{v},
\]
where the fibers over $S$ of $\mathcal X^\lambda_{v}$ are expansions of $X_v$. The scheme theoretic structure of this map is clear. In fact, it has a natural logarithmic structure. The curve $C_v$ has a logarithmic structure obtained by marking the preimages of nodes as marked points. The scheme $X_{v}$ has its divisorial a logarithmic structure, while the expansion $ \mathcal X^\lambda_{v}$ is a logarithmic modification of the trivial family of $X_{v}$. 

\begin{lemma}
The morphism $f_i$ is logarithmic, with the logarithmic structures given to curve and target as above.
\end{lemma}

\begin{proof}
The logarithmic structure at the curve $C_v$ and of the map away from the preimages of the normalized nodes is clear. At these marked points the curve is dimensionally transverse to the target, so pull back the divisorial logarithmic structure on $X_v$ to $C_v$. The resulting logarithmic structure on $C_v$ is generated in the neighborhood of a smooth point by a power of the local parameter at the marked point, with exponent equal to the contact order. This local parameter determines an element of the logarithmic structure sheaf of the curve, the identification defines the desired logarithmic structure on the morphism. 
\end{proof}

This procedures gives rise to a cutting morphism
\begin{equation}\label{eq: cutting}
\kappa: \Kbar^\lambda_\rho(Y_0) \to \prod_v \Kbar^\lambda_\rho(X_{v}).
\end{equation}
In order to glue, we apply Theorem~\ref{thm: var-on-main-comb-thm} to first modify the product. As we have done in the first part of the paper, the subdivision
\[
{\bigtimes_v}^\lambda \Kbar_{\rho}(\Sigma_{v}) \to \prod_v \Kbar^\lambda_{\rho}(\Sigma_{v})
\]
can be pulled back to produce a logarithmic modification of stacks
\[
{\bigtimes_v}^\lambda \Kbar_{\rho}(X_{v}) \to \prod_v \Kbar^\lambda_{\rho}(X_{v}). 
\]
The stack ${\bigtimes_v}^\lambda \Kbar_{\rho}(X_{v})$ is equipped with a natural virtual fundamental class, coming from the analogous modified product of maps to the Artin fan. For a cofinal system of choices $\lambda$, we also have a modified cutting morphism 
\[
\kappa: \Kbar^\lambda_\rho(Y_0)\to {\bigtimes_v}^\lambda \Kbar_{\rho}(X_{v}).
\]
We are guilty of a mild abuse of notation. The cutting morphism in Equation~(\ref{eq: cutting}) need not immediately lift to the modified product, but the source $\Kbar^\lambda_\rho(Y_0)$ can be modified to lift the cutting morphism. The modification is formed by the fiber product of the corresponding tropical moduli spaces, and therefore amounts to a different choice of $\lambda$ in the inverse system. Since we have not specified a particular choice of $\lambda$, we avoid overburdening the notation. 

\subsection{Virtual classes} Given a point $p$ of the modified product ${\bigtimes_v^\lambda} \Kbar_{\rho}(X_{v})$, the moduli space associates a disconnected universal curve, a glued target family, and a transverse map from each curve to a partial normalization of the target. Let $\mathcal Y_0^\lambda$ be the fiber of the target family over this chosen point of the product. Let
\[
\mathcal Y_0^{\lambda,\circ}[p]\subset \mathcal Y_0^{\lambda}[p]
\]
be the complement of the union of the closed codimension $2$ strata in $\mathcal Y_0^{\lambda}[p]$. In other words, this interior consists of the irreducible components and their locally closed divisorial boundaries. The codimension $1$ strata of $\mathcal Y_0^{\lambda,\circ}[p]$ are therefore smooth. 

The same works in families. For the expanded target family, the fiberwise complement of the codimension $2$ strata forms a flat family of open targets, and a family of smooth divisors. 

\begin{lemma}
Let $\mathcal Y_0^\lambda\to {\bigtimes_v^\lambda} \Kbar_{\rho}(X_{v})$ be the family of expanded, compatibly glued targets. The union of $Y_0^{\lambda,\circ}[p]$ ranging over $p$ in the base form a flat family. Furthermore, universal map
\[
\bigcup C_v\to \mathcal Y_0^{\lambda,\circ}
\]
is proper. 
\end{lemma}

\begin{proof}
Switching to the tropical picture, we have a map
\[
\bigcup \plC_v\to \Sigma_\infty.
\]
By the transversality hypothesis, every cone of $\bigcup \plC_v$ surjects onto a cone of the target $\Sigma_\infty$. This gives a union of cones, whose support determines an open subscheme $\mathcal Y_0^{\lambda,\circ}$ of the target family, which is flat. The properness of the universal map also follows by transversality: by design we have built target expansions such that the higher codimension strata are never met by the curve. 
\end{proof}

The smoothness of the divisors is of particular use to us here, and we will refer to them as \textbf{open divisors}. Recall that the edges $e$ of $G$ produce boundary evaluations on the modified product stack:
\[
\mathsf{ev_e}: {\bigtimes_v}^\lambda \Kbar_\rho(X_{v})\to \mathrm{D}^{\{2\}}_e(\mathcal X^\lambda).
\]
The symbol $\mathrm{D}_e(\mathcal X^\lambda)$ denotes the open divisor family supporting the evaluations of marked points coming from $e$. The evaluation above is to its fiber square. Let $\mathcal D_G$ be the fiber product of the open divisor families $\mathrm{D}_e(\mathcal X^\lambda)$, ranging over all nodal edges $e$ of $G$. Let $\mathcal D_G^2$ be its fiber square. We have an evaluation $\mathsf{ev}$ at points on each curve corresponding to the two flags of an edge $e$ in $G$. Packaging these data leads to a fiber square in ordinary schemes
\[
\begin{tikzcd}
\mathcal{G}^\lambda(Y_0)\arrow{r}{\mathrm{split}}\arrow{d} & {\bigtimes_v^\lambda} \Kbar_{\rho}(X_{v}) \arrow{d}{\mathsf{ev}} \\
\mathcal D_G \arrow[swap]{r}{\Delta} & \mathcal D^2_G.
\end{tikzcd}
\]
The fiber product parameterizes collections of logarithmic stable maps to expansions of $X_v$ whose underlying maps glue to form a map to the scheme $Y^\lambda_0$. In particular, the maps have no natural logarithmic structure at the gluing nodes. 

The space $\mathcal{G}^\lambda(Y_0)$ has a natural virtual fundamental class described as follows. Since the open divisors don't meet, it follows that $\mathcal D_G$ is smooth over the base. The morphism $\Delta$ is therefore a regular embedding and defines a perfect obstruction theory given by its normal bundle~\cite[Section~6.2]{Ful93}. We can pull back this obstruction theory to the top row. Define the class $\left[\mathcal{G}^\lambda(Y_0)\right]^{\mathrm{vir}}$ as the virtual fundamental class obtained by virtual pull back of the virtual class on the top right:
\[
\left[\mathcal{G}^\lambda(Y_0)\right]^{\mathrm{vir}} := \Delta^! \left[{\bigtimes_v}^\lambda \Kbar_\rho(X_{v})\right]^{\mathrm{vir}}. 
\] 

This class has a standard description by unpacking the obstruction theory provided by the diagonal, and Manolache's virtual pullback~\cite{Mano12} and a detailed description can be found for instance in~\cite[Section~7.4]{Bou19} or~\cite{KLR}. The upshort is the following. First let $\mathscr A_{v}$ be the Artin fan of $X_{v}$ and let ${\bigtimes_v^\lambda} \Kbar_\rho({\mathscr A}_{v})$ be the Artin fan version of the modified product -- pull back the tropical moduli space and its universal families to the product of the spaces of maps to the Artin fan. Consider the natural morphism from the cotangent complex of the diagonal:
\[
\mathsf{ev}^\star \mathrm L_\Delta\to\mathrm{split}^\star \boxtimes_v (\mathrm{R}\pi_{v\star} f_v^\star T_v^{\mathrm{log}})^\vee,
\]
where $T_v^{\mathrm{log}}$ is the logarithmic tangent bundle of the target $X_v$, and $\pi_v$ and $f_v$ are the universal curve and map respectively. The cone of this morphism defines an obstruction theory on the glued space, and thus an obstruction theory relative to ${\bigtimes_v^\lambda} \Kbar_\rho(\mathscr A_{v})$. Geometrically, the deformations are controlled by those logarithmic tangent fields of the target that preserve the gluing conditions.

\begin{remark}
The relative diagonal used above is \textit{not} the pullback of the diagonal from the projection to $D^2$, although it is contained inside of it. The difference between the two accounts for the tropical transversality requirement in the previous section. In the usual degeneration formula, this diagonal will have a K\"unneth decomposition in cohomology, and Poincar\'e duality can then be used to distribute insertions across the gluing markings. In this case, the diagonal does not have a K\"unneth decomposition and must be worked with directly. A similar point arises in work of Pandharipande and Pixton on the Gromov--Witten/Pairs correspondence for the quintic threefold, and the techniques developed there are likely to be useful here~\cite[\textit{Section~1.2}]{PP17}. In order to understand this conceptually, one can apply a similar and simpler version of the ideas in the present paper to the evaluation space -- the space of maps from $n$ points to a divisor $D$, which are not allowed to touch the boundary of $D$. This will produce a modification of the evaluation space constructed in~\cite{ACGM}, and we will construct it in a separate paper. 
\end{remark}

By performing a base change for the degeneration $\mathscr Y\to \A^1$, we assume that its special fiber is reduced. By modifying $\lambda$, we ensure that the universal target family is also reduced.

\begin{lemma}
The morphism that forgets the logarithmic structure at the gluing nodes
\[
\mu: \Kbar^\lambda_\rho(Y_0)\to \mathcal{G}^\lambda(Y_0)
\]
is finite and \'etale of degree $\prod_e m_e$, where $m_e$ is the expansion factor on the edge $e$ of the rigid tropical curve. 
\end{lemma}

\begin{proof}
This is a well known fact that has appeared in various places in the literature. To see that the map is surjective, note that every point in $\mathcal G^\lambda(Y_0)$ gives rise to a map to the special fiber of an expansion of the degeneration $\mathscr Y$:
\[
C_0^\lambda\to \mathcal Y_0^\lambda\hookrightarrow \mathscr Y^\lambda,
\]
such that $C^\lambda_0$ is transverse to the special fiber. In particular, the nodes of $C^\lambda_0$ map to locally closed codimension $1$ strata in $Y^\lambda_0$ where two components meet. The special fiber $Y^\lambda_0$ inherits a logarithmic structure from the degeneration $\mathscr Y^\lambda$, and since the curve is dimensionally transverse, the pullback of the logarithmic structure gives rise to a unique logarithmic lift in the category of fine but not necessarily saturated logarithmic schemes. The existence of lifts is immediate from transversality~\cite[Lemma~4.2.2]{AMW12}, while the uniqueness follows from~\cite{Wis16b}. 

It remains to count the number of lifts and show that it is the claimed number. The calculation is well known, and is recorded in complete generality and great detail in~\cite[Theorem~5.3.3]{ACGS15}. Note that~\cite[Proposition~7.1]{NS06} and~\cite[Lemma~7.9.1]{CheDegForm} both suffice here in the reduced case. 
\end{proof}

We have one final compatibility to check. Parallel to the tropical smoothing at the end of the previous section, we verify that once the map is transverse at the gluing nodes, there are no additional obstructions to promote a map to a logarithmic one. 

\begin{lemma}
There is an equality
\[
\mu_\star[\Kbar^\lambda_\rho(Y_0)]^{\mathrm{vir}}  = \prod_e m_e\cdot \left[\mathcal{G}^\lambda(Y_0)\right]^{\mathrm{vir}}.
\]
\end{lemma}

This fact has appeared in the literature in nearly identical form. The reader can find the proof in detail in~\cite[Section~7.10]{CheDegForm} and~\cite[Section~5.8.1]{AF11}. We follow~\cite[Section 6]{AW}.

\begin{proof}
Two approaches are available. The one used in~\cite{CheDegForm} is as follows. Since the morphism $\mu$ is \'etale, the obstruction theory on $\mathcal{G}^\lambda(Y_0)$ pulls back to an obstruction theory on $\Kbar^\lambda_\rho(Y_0)$. This space already has an obstruction theory, so we must check that the two obstruction theories coincide. Since our target is expanded, simply delete the codimension $2$ strata of $Y_0^\lambda$, such that over every point, the target is a union of smooth varieties meeting transversely along a divisor that is smooth in each. Now directly apply the calculation involving cotangent complexes in~\cite[Proposition~7.10.1]{CheDegForm} to conclude that the two obstruction theories coincide, from which the result follows immediately. 

We give the details on a second approach, seeing the compatibility of obstruction theories without reference to the cotangent complex by following~\cite[Section~7.3]{Wis11}. The standard obstruction theory for maps to $\mathscr Y$ is obtained by considering lifts over $\Kbar^\lambda_\rho(\mathscr A_{\mathscr Y})$:
\[ \vcenter{\xymatrix{
& & Y_0 \ar[d]  \\
C \ar[r] \ar@/^15pt/[urr]^f \ar[d] & C' \ar[r] \ar@{-->}[ur] \ar[d] & \mathcal A_{\mathscr Y}\ar[d] \\
S \ar[r] & S'\ar[r] &\spec(\NN\to\CC),
}} \]
where $S\subset S'$ is a strict square zero thickening of $S$. Here, the map $Y_0\to \mathcal Y$ is a strict logarithmic morphism. The lifts of this diagram form a torsor under the sheaf of abelian groups $f^\star T_{Y_0}^{\mathrm{log}}\otimes I_{S\subset S'}$.  The stack on $S$ of such torsors on $C$ defines an obstruction theory on the space of maps, relatively over $\Kbar^\lambda_\rho(\mathcal A_{\mathscr Y})$. 

On the other hand, the obstruction theory pulled back from $\mathcal{G}^\lambda(Y_0)$ by definition considers the deformation problem for logarithmic maps from partial normalizations of a logarithmic curve, to expansions of the components $X_v$ with their divisorial logarithmic structure that are required to glue scheme theoretically to form a map to the underlying scheme of $Y_0$. 

The lifting problem for this second obstruction theory is obtained from diagrams:
\[ \vcenter{\xymatrix{
& & Y^\lambda_0 \ar@/^15pt/[ddl]  \\
C \ar[r] \ar@/^15pt/[urr]^f \ar[d] & C' \ar@{-->}[ur] \ar[d]  \\
S \ar[r] & S'.
}} \]
Over $S'$ there is a canonical closed embedding of schemes, including the special fiber as
\[
\mathcal Y_0^\lambda\to \mathscr Y^\lambda
\]
where $\mathscr Y^\lambda\to \A^1$ is the relative degeneration over $S'$ chosen in the construction of the modified product ${\bigtimes_v^\lambda} \Kbar_\rho(X_{v})$. The map $C\to \mathscr Y_0^\lambda$ is transverse to the strata, and since the extension is strict, any lift from $C'$ is automatically transverse. Endow $ \mathcal Y_0^\lambda$ with the logarithmic structure coming from its inclusion as the special fiber of the degeneration. The transversality guarantees that this map lifts to a logarithmic map, compatible with the given logarithmic structures on $C$ and $C'$.

Conversely, given a lift of the first diagram, simply forget the logarithmic structure at the marked nodes to obtain a lifting diagram below. This canonically identifies the torsors defining the obstruction theories. The lifts of the second diagram therefore also form a torsor under the sheaf of abelian groups $f^\star T_{Y_0}^{\mathrm{log}}\otimes I_{S\subset S'}$. This identifies the obstruction theories, and virtual classes, as required. 
\end{proof}

\subsection{Implementation and special cases}\label{sec: implementations} We summarize the previous sections and put the pieces together for easy access. The gluing formula is obtained as follows. For a rigid tropical stable map $[\rho]$ with underlying graph $G$, we construct a virtual birational modification of the product of moduli spaces of maps over the vertices. This space is equipped with boundary evaluations corresponding to flags of edges in $G$. The target of the boundary evaluation at $e$ is generically the divisor $D_e$ dual to $e$. Over deeper strata in the modified product, the target is replaced by toric bundles over the strata of $D_e$. The gluing fiber diagram is the following one:
\[
\begin{tikzcd}
\mathcal{G}^\lambda(Y_0)\arrow{r}{\mathrm{split}}\arrow{d} & {\bigtimes_v^\lambda} \Kbar_{\rho}(X_{v}) \arrow{d}{\mathsf{ev}} \\
\mathcal D_G \arrow[swap]{r}{\Delta} & \mathcal D^2_G.
\end{tikzcd}
\]
The two vertical arrows are sections of smooth fibrations. Up to the combinatorial factor $\prod_e m_e$, the virtual count attached to $[\rho]$ is computed by a virtual integral on $\mathcal{G}^\lambda(Y_0)$. Our degeneration formula for the virtual class involves a relative diagonal, i.e. the class of the fiberwise diagonal in the fiber square of the divisor expansion. This complicates numerical applications of the degeneration formula. An identical phenomenon occurs in Parker's theory. Parker's refined cohomology, which supports his cycle theoretic degeneration formula, does not satisfy a K\"unneth formula, and hence, no diagonal splitting. Nonetheless, our formula collapses in special cases. 

\subsubsection{Examples in the existing literature} A class of examples are present in the literature already. As has already been mentioned, when the degeneration $\mathscr Y$ has no triple of components intersecting, the formula above collapses immediately to Jun Li's formula. Examples beyond this case can be found in surface geometries. In the papers~\cite{Bou19,CJMR,GPS}, explicit arguments show that the relevant Gromov--Witten cycles can be supported away from loci where the universal divisor expands nontrivially. These arguments are recorded in~\cite[Proposition~11]{Bou19},~\cite[Proposition~5.10]{CJMR}, and~\cite[Proposition~4.2]{GPS} in the appropriate contexts. It follows that the relative and ordinary diagonals coincide for the purpose of calculating invariants, and a degeneration formula follows. On the other hand, when dealing with unobstructed geometries, such as genus $0$ logarithmic descendant invariants of toric varieties, the moduli spaces are zero dimensional, and again, no additional divisor bubbling arises. This provides a direct link to work Mandel and Ruddat~\cite{MR16}, and earlier work of Nishinou and Siebert~\cite{NS06}.

%
%

\subsubsection{The triple point degeneration in genus $0$} There is a simplified gluing formula that can be used for genus $0$ for degenerations where the dual complex has dimension $2$, and in particular, the divisors are smooth pairs. 

When the target is a toric surface degeneration, this simplification is implicit throughout the tropical literature, see for instance~\cite{GM07c}. It is also explained by Parker in his framework in~\cite[Section~4]{ParTriple}. 

Let $\mathscr Y\to \mathbb A^1$ be a simple normal crossings degeneration with generic fiber $Y_\eta$ and special fiber $Y$. Assume also that no four irreducible components of the special fiber $Y_0$ meet. The fibers of the associated tropical degeneration $\Sigma\to \RR_{\geq 0}$ are therefore at most $2$-dimensional polyhedral complexes. 

We describe an algorithm for computing the Gromov--Witten theory of $Y_\eta$ in genus $0$ with stationary insertions using the degeneration above.  Fix $n$ points in $Y_\eta$ and let $p_1,\ldots, p_n$ denote their limits in the special fiber. The tropicalizations of these points give rise to tropical point constraints $\{p_i^{\trop}\}$ in $\Sigma_\eta$. Fix a rigid tropical stable map $\gamma = [f_\eta:\plC_\eta\to \Sigma_\eta]$ meeting the point constraints $\{p_i^{\trop}\}$. We compute the $0$ cycle defined, for any indexing choice $\lambda$, as
\[
\langle p_1,,\ldots,p_n\rangle^{Y_0}_{\gamma} :=\int_{[\Kbar^\lambda_\gamma(Y_0)]^{\mathrm{vir}}} \prod_{i=1}^n \mathsf{ev_i}^\star(p_i) = \int_{[\mathsf{ACGS}_\rho(Y_0)]^{\mathrm{vir}}} \prod_{i=1}^n \mathsf{ev_i}^\star(p_i).
\]

In this circumstance, observe that the irreducible component to which the evaluation morphism attached to an edge $e$ maps is always $D_e$, rather than a bundle over a stratum of it. This follows from the fact that $D_e$ has a $1$-dimensional cone complex. The genus $0$ condition precludes the possibility of parallel edges in $G$. These facts combine to simplify the gluing formula in this case to the following algorithm. 

\begin{algorithm}
The invariant $\langle p_1,,\ldots,p_n\rangle^{Y_0}_{\gamma}$ is computed by the following procedure.
\begin{enumerate}[{\bf Step A:}]
\item Orient the graph $\plC_\eta$ such (1) that all ends of contact order $0$ are incoming to their incident vertex, and (2) every vertex has at most one outgoing edge. 
\item For each vertex $v$ of $\plC_\eta$, let $\mathsf{K}_v$ be the moduli space of maps to expansions of the component $Y_v$ with type given by the star of $v$ in $[\plC_\infty\to \Sigma_\infty]$. 
\item For an edge $e$, let $D_e$ be the stratum to which it is dual. Let $v$ be a vertex with exiting edge $e_v$, define the operator on cohomology
\[
\mathsf{GW_v}: \bigotimes_{e:\mathrm{incoming}} H^\star(D_e;\mathbb Q)\to H^\star(D_{e_v};\mathbb Q)
\]
as follows. Given a class in $\bigotimes_{e:\mathrm{incoming}} H^\star(D_e;\mathbb Q)$, pull it back via the evaluation morphisms to the moduli space $\mathsf{K_v}$, cap it with $[\mathsf{K_v}]^{\mathrm{vir}}$, and push it forward via the evaluation morphism $\mathsf{ev_{e_v}}$. 
\item Assign incidence cycles to edges $e$ as follows. If $e$ is an edge attached to a marked point with trivial contact, then define a class $\alpha(e)$ to be the class of a point. Otherwise, if $e$ is the exiting edge of a vertex $v$, define 
\[
\alpha(e) = \mathsf{GW}_v(\otimes_{e'}\alpha(e')),
\]
where $e'$ ranges over the incoming edges incident to $v$. 
\end{enumerate}
\end{algorithm}

\subsubsection{Validation of the algorithm} We provide a sketch that the simplified algorithm computes the invariants at hand. Fix an edge $e$ of the curve $\plC_\infty$. If we cut the graph $\plC_\infty$ at $e$, we obtain two tropical curves. Let $\mathsf K_1$ and $\mathsf K_2$ denote the moduli spaces associated to these two tropical curves. According to the main theorem, we have the following diagram
\[
\widetilde{\mathsf K_{1}\times\mathsf K_{2}} \to \mathrm D_e^{\{2\}}\to D_e\times D_e,
\]
where the first map is a section of a smooth fibration. Note that by construction $\mathrm D_e^{\{2\}}$ is an open subset of a logarithmic modification associated to the product:
\[
\mathrm D_e^{\{2\}}\to {\mathsf K_{1}\times\mathsf K_{2}} \times D_e^2.
\]
The second arrow above is obtained by blown down and projection onto the second factor in this product. By passing to a further subdivision, we can assume that there is a factorization
\[
\widetilde{\mathsf K_{1}\times\mathsf K_{2}} \xrightarrow{\vartheta} \widetilde{D_e\times D_e}\xrightarrow{\varpi} D_e\times D_e,
\]
such that $\vartheta$ is combinatorially flat and $\widetilde{D_e\times D_e}$ is smooth. The degeneration formula proved previously shows that the virtual class of the moduli space associated to $\gamma$ can be computed in two steps: (i) take the diagonal $D_e$ inside $D_e\times D_e$ and calculate its $\varpi$-strict transform, and (ii) pullback the resulting Chow cohomology class via $\vartheta$ and operate it on the virtual class. 

We now use a simple observation. The morphism
\[
\mathsf K_1\times\mathsf K_2\to D_e\times D_e
\]
is \textit{automatically} combinatorially flat. Indeed, the morphism is a product of two morphisms, and it suffices to check the combinatorial flatness on each factor. Since the tropical target has cones of dimension $0$ or $1$, the condition is immediate. Now consider the following intersection diagram:
\[
\begin{tikzcd}
\widetilde{\mathsf K_{1}\times\mathsf K_{2}}\arrow{dr}  & & \\
 & (\mathsf K_{1}\times\mathsf K_{2})^\dagger \arrow{d}\arrow{r} & \mathsf K_{1}\times\mathsf K_{2}\arrow{d} \\
 & \widetilde{D_e\times D_e}\arrow{r} & {D_e\times D_e}. 
\end{tikzcd}
\]
The square is defined by fine and saturated logarithmic base change. However, the right vertical arrow is combinatorially flat, so the square is fiber in the category of algebraic stacks as well. Since the diagonal in $\widetilde{D_e\times D_e}$ pushes forward to the diagonal in $D_e\times D_e$, it follows that the pushforward of the virtual class $[\mathsf K_\gamma]$ to $\mathsf K_1\times\mathsf K_2$ is precisely the pullback of the diagonal in $D_e\times D_e$ under the right vertical arrow above. 

The algorithm is now justified for a single edge $e$ by an elementary diagram chase; the general case follows by induction. \qed

\bibliographystyle{siam} 
\bibliography{DegenerationFormula2022}

\end{document}